\documentclass[12pt]{amsart}

\usepackage{amsmath,amssymb,amsfonts,fullpage,amssymb,amsthm}
\usepackage{latexsym}
\usepackage{mathrsfs}
\usepackage{mathtools}
\usepackage{wasysym}
\usepackage{dsfont}

\newtheorem{thm}{Theorem}
\newtheorem{prop}{Proposition}
\newtheorem{lem}{Lemma}
\newtheorem{defn}{Definition}
\newtheorem{coro}{Corollary}
\newtheorem{claim}{Claim}
\newtheorem*{clm}{Claim}

\newtheorem*{notation}{Notation}

\def\ninf{\mathbb N^{[\infty]}}

\def\fininf{\text{FIN}^\infty}
\def\finkinf{\text{FIN}_k^{\infty}}
\def\finfin{\text{FIN}^{<\infty}}
\def\finkfin{\text{FIN}_k^{<\infty}}

\newcommand{\cU}{\mathcal{U}}
\def\Mh{\mathbb M_{\mathcal H}}
\def\col{Col(\omega , <\lambda)}

\newcommand{\fink}{\text{FIN}_k}
\newcommand{\finkk}{\text{FIN}_k^\infty}
\newcommand{\less}{<}

\newcommand{\linfk}{\cA\finkk}
\newcommand{\one}{\mathds{1}}

\newcommand{\bM}{\mathbb{M}}

\newcommand{\cA}{\mathcal{A}}
\newcommand{\cD}{\mathcal{D}}

\newcommand{\cH}{\mathcal{H}}

\newcommand{\cX}{\mathcal{X}}
\newcommand{\cY}{\mathcal{Y}}

\DeclareMathOperator{\depth}{depth}

\DeclareMathOperator{\FIN}{FIN}

\title{Ramsey subsets of the space of infinite block sequences of vectors .}

\address{* Instituto Venezolano de Investigaciones Cient\'{\i}ficas\\
**Universidad de Los Andes, Bogot\'a, Colombia \\*** University of Colorado Denver. }

\author{Daniel Calder\'on $^{**}$, Carlos Augusto Di Prisco$^{*, **}$ Jos\'e Gregorio Mijares$^{***}$ }

\begin{document}



\begin{abstract}

We study families of infinite block sequences of   elements of the space $\FIN_k$. In particular we study Ramsey properties of such families and Ramsey properties localized to a selective or semiselective coideal. We show how the stable ordered-union ultrafilters defined by Blass, and Matet-adequate families defined by Eisworth in the case $k=1$ fit in the theory of the Ramsey space of infinite block sequences of finite sets of natural numbers.

\end{abstract}

\maketitle

\section{Introduction}

In this article we study the Ramsey property of subsets of the space $\finkinf$ of infinite block sequences of elements of $\FIN_k$. The case $k=1$ deals with block sequences of finite sets of natural numbers. We consider coideals contained in these spaces and  the Ramsey property localized on such a coideals. 

Let $k$ be a positive integer, 
$\FIN_k=\{p: \mathbb N\to \{0, 1, \dots , k\}: \{n: p(n)\neq 0\} \mbox{ is finite and  } k\in \mbox{ range} (p)\}$.
For $p\in \FIN_k$, supp$(p)= \{n: p(n)\neq 0\}$.
$\FIN_k$ is a partial semigroup under the partial semigroup operation of addition of elements with disjoint support.
A block sequence of elements of $FIN_k$ is a (finite or infinite) sequence $(p_n)$ with supp$(p_n)<$ supp$(p_{n+1})$ for every $n\in \mathbb N$ (i.e. the maximal element of supp$(p_n)$ is strictly below the minimal element of supp$(p_{n+1})$ for every $n$).

The relation between $\FIN_k$ and the positive part of the unitary sphere of the Banach space $c_0$ is well known, see for example \cite{To}, page 37. It permits to identify  elements of $\FIN_k$ with vectors.
 
The operation $T:\FIN_k\to \FIN_{k-1}$ is defined by $T(p)(n)=\mbox{max}\{p(n)-1, 0\}$.

Given an infinite block sequence $A=(p_n)$ of elements of $\FIN_k$, the subsemigroup $[A]$ of $\FIN_k$ generated by $A$ is the collection of elements of $\FIN_k$ of the form
$$T^{(i_0)}(p_{n_0})+ \dots +T^{(i_l)}(p_{n_l})$$
for some sequence $n_0<\dots <n_l$  and some choice $i_0, \dots , i_l\in \{0, 1, \dots , k\}$. Notice that for the sum to remain in $\FIN_k$ at least one of the numbers 
$i_0, \dots , i_l$ must be $0$.

We denote by $\finkinf$ the space of infinite block sequences of elements of $\FIN_k$. If $k=1$ we simply write $\fininf$; in this case, $[A]$ is the subsemigroup of $\fininf$ formed by all the block sequences whose elements are finite unions of elements of $A$.

To define the Ramsey property of subsets of $\finkinf$, we first recall the definition of the Ramsey property for subsets of  the space $\ninf$ of all infinite sets of natural numbers.  With the product topology (the topology inherited from the product topology on $2^{\mathbb N}$), this space  is homeomorphic to $\mathbb R\setminus \mathbb Q$, the irrational numbers.
The exponential topology of this space, also called the Ellentuck topology,  is finer than the product topology and it is generated by the basic sets of the form
$$[a,A]=\{X\in \ninf : a\sqsubset X \subseteq  A\}, $$
where $a$ is a finite set of natural numbers, $A$ is an infinite subset of $\mathbb N$, and $a\sqsubset X$ means that $a$ is an initial segment of $X$ in its increasing order.

A subset $\mathcal A\subseteq \ninf$ is Ramsey, or has the Ramsey property, if for every $[a,A]$ there is an infinite subset $B$ of $A$ such that $[a,B]\subseteq \mathcal A$ or $[a,B]\cap \mathcal A=\emptyset$.
Silver proved that all analytic subsets of $\ninf$ have the Ramsey property. His proof  has  a metamathematical character, as opposed to the combinatorial proof of Galvin and Prikry for the Borel sets. Ellentuck \cite{Ell} gave a topological proof of Silver's result by showing that  a subset of $\ninf$ is Ramsey if and only if it has the property of Baire with respect to the exponential topology.

We present below similar results for the space $\finkinf$ of infinite block sequences of elements of $\FIN_k$.
We also consider certain subfamilies of $\finkinf$ to define coideals, selective coideals and semiselective coideals; and study some forcing notions related to these subfamilies.
Previous work in this subject was done in \cite{Bl, Ei, Ga, Mat}.
More recently, Zhang \cite{Zh} studies the preservation of selective ultrafilters on $\FIN$ under Sacks forcing, proves that selective ultrafilters on $\FIN$ localize ther parametrized Milliken theorem,   and also proves that  that those selective ultrafilters are Ramsey.

Garc\'{\i}a \'Avila, in \cite{Ga}, considers several forcing notions related to the space $\fininf$,  and in particular a forcing notion analogous to Mathias forcing  adapted to this space.
She proves that this notion has a pure decision property (a Prikry property) and asks if it has a property analogous to the fact that an infinite subset of a Mathias generic real is also a Mathias generic real (hereditary genericity, or the Mathias property).
This question was answered positively in \cite{CaDP}, and here we extend this answer to the forcing localized on a semiselective coideal.

In this article we study these forcing notions and their relation to some classes of ultrafilters introduced by Blass and Hindman (see \cite{Bl}). Stable ordered-union ultrafilters on the space $\FIN$ of finite sets of natural numbers were defined by Blass (\cite{Bl}); these ultrafilters are related to Hindman's theorem on partitions of $\FIN$ in the same way selective ultrafilters on $\omega$ are related to Ramsey's theorem. We show that stable ordered-union ultrafilters are closely related to selective ultrafilters on the Ramsey space $\fininf$.
In his study of forcing and stable ordered-union ultrafilters (\cite{Ei}) Eisworth isolates the concept of Matet-adequate families of elements of $\fininf$, and proves that forcing with such a familiy adds a stable ordered-union ultrafilter. We show that Matet-adequate families correspond to selective coideals of the topological Ramsey space $\fininf$.

%In particular we consider selective and semiselective coideals ultrafilters ??????  on the space FIN which have been previously studied in \cite{Zh}. 
We also address the problem of the consistency of the statement all subsets of $\finkinf$ have the Ramsey property localized with respect to a semiselective coideal.
This presentation is formulated in the context of topological Ramsey spaces as presented in  \cite{To}.

\section{Block sequences of elements of $\FIN_k$.}\label{block}
%Let $FIN$ be the collection of non-empty finite sets of natural numbers. 

As defined above, for a positive integer $k$,  $\finkinf$  denotes the collection of all infinite block sequences of elements of $FIN_k$, that is to say, all sequences $p_0, p_1 \dots$ where, for all $i\in \omega$, $p_i\in \FIN_k$ and $\mbox{max(supp}(p_i))<\mbox{min(supp}(p_{i+1}))$.

For a positive integer $d$, $\fink^{[d]}$ denotes the collection of all finite block sequences of length $d$.

We define the \textbf{approximation space of} $\finkinf$ as the set
$$\linfk:=\bigcup\left\{\fink^{[d]}:d\in\omega\right\}= \FIN_k^{<\infty}$$
 the collection of finite block sequences of elements of $\FIN_k$. 
 
 For each $m\in\omega$, we define the \textbf{approximation function} $r_m:\finkk\to\fink^{[m]}$ that sends an infinite block sequence to its first $m$ blocks. We define $r:\finkinf\times \omega\to \finkfin$ by $r(X,n)=r_n(X)$.

For $A\in\finkk$ we use the symbols
$$\finkk\upharpoonright A:=\{B\in\finkk:B\subseteq[A]\}\text{ and }\finkfin\upharpoonright A:=\{s\in\finkfin:s\subseteq[A]\}$$

% The metric topology on $\FIN_k^\infty$ is generated by the sets of the form $$[s]= \{X\in \FIN_k^\infty : s\sqsubset X\},$$ where $s\in \FIN_k^{<\infty}$ and $s\sqsubset X$ means that $s$ is an initial segment of $X$.

Consider the binary relation defined on $\finkinf$ by
$X\leq Y$ if $X$ is a {\bf condensation} of $Y$, that is, every element of $X$ belongs to the subsemigroup $[Y]$ generated by  $Y$.

%Consider also the functions $r_n: \finkinf\to \FIN_k^n$ given by seting $r_n(X)$ equal to the finite block sequence of the first $n$ elements of $X$. 
The triple $(\FIN_k^\infty, \leq, r)$ satisfies the following  properties A1-A4.
\bigskip

\noindent \textbf{(A.1)}\ [{\bf Metrization}] 
\begin{itemize}
	\item[{(A.1.1)}]For any $A\in \finkinf$, $r_{0}(A) = \emptyset$.
	\item[{(A.1.2)}]For any $A,B\in \finkinf$, if $A\neq B$ then
	$(\exists n)\ (r_{n}(A)\neq r_{n}(B))$.
	\item[{(A.1.3)}]If $r_{n}(A) =r_{m}(B)$ then $n = m$ and $(\forall i<n)\ (r_{i}(A) = r_{i}(B))$.
\end{itemize}

Take the discrete topology on $\finkfin$ and endow 
${(\finkfin)}^{\mathbb{N}}$ with the product topology; this is the metric space of all the sequences of 
elements of $\finkfin$.  Notice  that $\finkinf$ is a closed
  subspace of ${(\finkfin)}^{\mathbb{N}}$. 
  
 With this notation, the basic open sets generating the metric topology on $\finkinf$ are of the form
\begin{equation}\notag\label{eq basic opens in ARN}
		[s] = \{B\in \finkinf : (\exists n)(s =
	r_{n}(B))\}
\end{equation}
where $s\in\finkfin$.  Let us define the {\bf length} of $s$,
as the unique integer $|s|=n$ such that $s = r_{n}(A)$ for some $A\in
\finkinf$.

We will consider another topology on $\finkinf$ which we will call the Ellentuck (or exponential) topology. The {\bf Ellentuck type neighborhoods} are of the form:
\begin{equation}\notag\label{eq basic Ellentuck opens}
		[a,A] = \{B\in [a] : B\leq A\}=
		\{B\in \finkinf : (\exists n)\ a=r_n(B) \ \&\ B\leq A\}.
\end{equation}
where $a\in\finkfin$
and $A\in \finkinf$. 

\medskip

We will use $[n,A]$ to abbreviate $[r_{n}(A),A]$. 

\medskip

Notice that 
\begin{equation}\notag
    \finkfin\upharpoonright  A = \{a\in\finkfin : [a,A]\neq\emptyset\}.
\end{equation}
Given a neighborhood $[a,A]$ and $n\geq |a|$, let $r_n[a,A]$ be the image of $[a,A]$ by the function $r_n$, i.e., 
\begin{equation}\notag\label{eq image of basic Ellentuck opens}
		r_n[a,A] =
		\{r_n(B) : B\in [a,A] \}.
\end{equation}

\medskip

Given $a,b\in\finkfin$, write 

\begin{equation}\notag
a\sqsubseteq b\ \mbox{ iff } (\exists A\in\finkinf)\ (\exists m,n\in\mathbb N)\  m\leq n, a=r_m(A)\mbox{ and } b= r_n(A).
\end{equation}

\medskip

By A.1, $\sqsubseteq$ can be proven to be a partial order on $\finkfin$.

\medskip
The relation $\leq_{fin}$ on $\finkfin$ is defined in a similar way as the relation $\leq$ on $\finkinf$, and we have the following.

\noindent \textbf{(A.2)}\ [{\bf Finitization}] The quasi order $\leq_{fin}$ on
$\finkfin$ satisfies:
\begin{itemize}
    \item[(A.2.1)] $A\leq B$ iff
    $(\forall n)\ (\exists m) \ \ (r_{n}(A)\leq_{fin} r_{m}(B))$.
    \item[(A.2.2)] $\{b\in \finkfin : b\leq_{fin} a\}$ is finite, for every
    $a\in \finkfin$.
\item[(A.2.3)] If $a\leq_{fin} b$ and $c \sqsubseteq a$ then there is $d \sqsubseteq b$ such that $c \leq_{fin} d$.
\end{itemize}

\medskip

Given $A\in\finkinf$ and $a\in\finkfin\upharpoonright  A$,
we define the {\bf depth} of $a$ in $A$ as 
\begin{equation}\notag\label{eq depth of a segment}
		\operatorname{depth}_{A}(a): = \min\{n :a\leq_{fin}r_n(A)\}
\end{equation}

\noindent \textbf{(A.3)}\ [{\bf Amalgamation}] Given $a$ and $A$
with $\operatorname{depth}_{A}(a)=n$, the following holds:
\begin{itemize}
		\item[(A.3.1)] $(\forall B\in [n,A])\ \ ([a,B]\neq\emptyset)$.
		\item[(A.3.2)] $(\forall B\in [a,A])\ \ (\exists A'\in [n,A])\ \
		([a,A']\subseteq [a,B])$.
\end{itemize}

\bigskip

\noindent \textbf{(A.4)}\ [{\bf Pigeonhole Principle} (essentially Gowers' Theorem \cite{Go}, Hindman's Theorem \cite{Hin} for $k=1$)] Given $a$
and $A$ with $\operatorname{depth}_{A}(a) = n$, for every $\mathcal{O}\subseteq \FIN_k^{|a|+1}$ there is $B\in
[n,A]$ such that $r_{|a|+1}[a,B]\subseteq\mathcal{O}$ or $r_{|a|+1}[a,B]\subseteq\mathcal{O}^c$.

\begin{defn}
A set $\mathcal{X}\subseteq \finkinf$
is \textbf{Ramsey} if for every neighborhood $[a,A]\neq\emptyset$
there exists  $B\in [a,A]$ such that $[a,B]\subseteq \mathcal{X}$
or $[a,B]\cap \mathcal{X} = \emptyset$. A set
$\mathcal{X}\subseteq \finkinf$ is \textbf{Ramsey null} if for
every neighborhood $[a,A]$ there exists  $B\in [a,A]$ such that
$[a,B]\cap \mathcal{X} = \emptyset$.

A set $\mathcal{X}\subseteq \finkinf$
has the \textbf{abstract Baire property} if for every neighborhood $[a,A]\neq\emptyset$
there exists  $\emptyset\neq[b,B]\subseteq [a,A]$ such that $[b,B]\subseteq \mathcal{X}$
or $[b,B]\cap \mathcal{X} = \emptyset$. A set
$\mathcal{X}\subseteq \mathcal{R}$ is \textbf{nowhere dense} if for
every neighborhood $[a,A]$ there exists  $[b,B]\subseteq [a,A]$ such that
$[b,B]\cap \mathcal{X} = \emptyset$.
\end{defn}

\begin{thm}(\cite{To})\label{fininfram}
$(\finkinf, \leq, r)$ is a topological Ramsey space.
In other words,  a subset $\mathcal X\subseteq \finkinf$ is Ramsey if and only if it has the Baire property with respect to the Ellentuck topology, and Ramsey null sets coincide with nowhere dense sets.
\end{thm}

\begin{proof}
The result follows from  the fact that $(\finkinf, \leq, r)$ satisfies A1, A2. A3 and A4 (see \cite{To}). 
\end{proof}

\begin{defn}
Given $\mathcal H\subseteq\finkinf$, we say that $\mathcal H$ is a \textbf{coideal} if it satisfies the following:
\begin{itemize}
\item[(a)] $\mathcal H$ is closed under finite changes, i.e. if $A\in \mathcal H$ and $A\triangle B$ is finite, then $B\in \mathcal H$.
\item[{(b)}] For all $A,B\in\finkinf$, if $A\in\mathcal H$ and $A\leq B$ then $B\in\mathcal H$.
\item[{(c)}] (${\bf A3}\mod\mathcal H$) For all $A\in\mathcal H$ and $a\in\finkfin\upharpoonright A$, the following holds:
\begin{itemize}
\item $[a,B]\neq\emptyset$ for all $B\in [\depth_A(a),A]\cap\mathcal H$.
\item If $B\in\mathcal{H}\upharpoonright A$ and $[a,B]\neq\emptyset$ then there exists $A'\in [\depth_A(a), A]\cap\mathcal H$ such that  $\emptyset\neq [a,A']\subseteq [a,B]$.

\end{itemize}
\item[{(d)}]  (${\bf A4}\mod\mathcal H$) Let $A\in\mathcal H$ and $a\in\finkfin\upharpoonright  A$ be given. For all $\mathcal O \subseteq \FIN_k^{|a|+1}$ there exists $B\in [\depth_A(a),A]\cap\mathcal H$ such that $r_{|a|+1}[a,B]\subseteq\mathcal O$ or $r_{|a|+1}[a,B]\cap\mathcal O=\emptyset$.
\end{itemize}

\end{defn}

For a coideal $\mathcal H$, and $A\in \finkinf$, 
$$\mathcal{H}\upharpoonright  A = \{B\in \cH: B\leq A\}.$$

The Ramsey property and the Baire property are localized to a coideal in the following fashion.

\begin{defn} $\mathcal X\subseteq \finkinf$ is \textbf{$\mathcal H$-Ramsey} if for every $[a,A]\neq \emptyset$, with $A\in \mathcal H$, there exists $B\in [a,A]\cap \mathcal H$ such that $[a,B]\subseteq \mathcal X$ or $[a,B]\subseteq \mathcal X^c$. If for every $[a,A]\neq \emptyset$, there exists $B\in [a,A]\cap\mathcal H$ such that $[a,B]\subseteq \mathcal X^c$; we say that  $\mathcal X$ is \textbf{$\mathcal H$-Ramsey null}.
\end{defn}

\begin{defn} $\mathcal X\subseteq \finkinf$ is $\mathcal H$-\textbf{Baire} if for every $[a,A]\neq \emptyset$, with $A\in \mathcal H$, there exists $\emptyset\neq[b,B]\subseteq[a,A]$, with $B\in\mathcal H$, such that $[b,B]\subseteq \mathcal X$ or $[b,B]\subseteq \mathcal X^c$. If for every $[a,A]\neq \emptyset$, with $A\in\mathcal H$, there exists $\emptyset\neq[b,B]\subseteq[a,A]$, with $B\in \mathcal H$, such that $[b,B]\subseteq \mathcal X^c$; we say that $\mathcal X$ is $\mathcal H$-\textbf{meager}.
\end{defn}

It is clear that if $\mathcal X\subseteq \finkinf$ is $\mathcal H$-Ramsey then $\mathcal X$ is $\mathcal H$-Baire.  
Theorem \ref{fininfram} can be extended to the properties $\mathcal H$-Ramsey and $\mathcal H$-Baire when the family $\mathcal H$ is a semiselective coideal. We define this notion now.

\begin{defn}\label{diagonal3}
 Given $A\in\finkinf$ and a sequence $\mathcal{A}=(A_n)_{n\in\mathbb N}\subseteq\finkinf$, we say that $B\in\finkinf$ is a \textbf{diagonalization} of $\mathcal A$ whitin $A$ if for every $b\in\finkfin\upharpoonright B$ with $\depth_A(b)=n$ we have $[b, B]\subseteq [b,A_n]$. 
\end{defn}

%\begin{defn}\label{diagonal3}
%Given $A\in\finkinf$ and a sequence $\mathcal{A}=(A_s)_{s\in \finkfin \upharpoonright A}\subseteq\finkinf$, we say that $B\leq A$ is a \textbf{diagonalization} of $\mathcal A$ whitin $A$ if for every $b\in\finkfin\upharpoonright B$ we have $[b, B]\subseteq [b,A_b]$. 
%\end{defn}

\begin{defn}\label{selective}
 A coideal $\mathcal H\subseteq\finkinf$ is \textbf{selective} if given $[a,A]\neq\emptyset$ with $A\in \mathcal H$, for every sequence $\mathcal{A}=(A_n)_{n\in\mathbb N}\subseteq\mathcal{H}\upharpoonright \!\!A$ such that $A_n\geq A_{n+1}$ and $[a,A_n]\neq\emptyset$, there exists $B\in\mathcal H\cap [a,A]$ which diagonalizes $\mathcal{A}$ within $A$.
\end{defn}

%\begin{defn}\label{selective}
 %A coideal $\mathcal H\subseteq\finkinf$ is \textbf{selective} if given %$[a,A]\neq\emptyset$ with 
% $A\in \mathcal H$, for every sequence $\mathcal{A}=(A_a)_{a\in \finkfin\upharpoonright \, A}\subseteq\mathcal{H}\upharpoonright  \,A$  there exists $B\in\mathcal{H}\upharpoonright  A$ which diagonalizes $\mathcal{A}$ within $A$.
%\end{defn}
Notice that  $\finkinf$ is a selective coideal.

 We will see that, for the case $k=1$, stable ordered-union ultrafilters  (see definitions \ref{uultra}, \ref{souultra}) give rise to other examples of selective coideals on $FIN$.

\begin{defn} Let $\mathcal H\subseteq\finkinf$ be  a coideal. Given sets $\mathcal D, \mathcal S\subseteq\mathcal  H$, we say that  $\mathcal D$ is \textbf{dense open in} $\mathcal S$ if the following hold:
\begin{enumerate}
\item $(\forall A\in \mathcal S)\ (\exists B\in\mathcal D)\ B\leq A$.
\item  $(\forall A\in \mathcal S)\ (\forall B\in\mathcal D)\ [A\leq B \rightarrow A\in\mathcal D]$.
\end{enumerate}
\end{defn}

\begin{defn}\label{diagonal2}
 Given $A\in\mathcal H$ and a collection $\mathcal D = \{\mathcal D_a\}_{a\in \finkfin\upharpoonright A}$ such that each $\mathcal D_a$ is dense open in $\mathcal H\cap [\depth_A(a),A]$, we say that $B\leq A$ is a \textbf{diagonalization} of $\mathcal D$ if there exists a family $\mathcal A=\{A_a\}_{a\in \finkfin\upharpoonright A}$, with $A_a\in\mathcal D_a$, such that for every $a\in\finkfin\upharpoonright B$ we have $[a, B]\subseteq [a,A_a]$. 
\end{defn}

\begin{defn}\label{semiselective}
We say that a  coideal $\mathcal H\subseteq\finkinf$ is \textbf{semiselective} if for every $A\in \mathcal H$, every collection $\mathcal D = \{\mathcal D_a\}_{a\in \finkfin\upharpoonright A}$ such that each $\mathcal D_a$ is dense open in $\mathcal H\cap [\depth_A(a),A]$ and every $C\in\mathcal{H}\upharpoonright  A$, there exists $B\in\mathcal{H}\upharpoonright  C$ such that $B$ is a diagonalization of $\mathcal D$ 
\end{defn}

\begin{lem}\label{filtered}
Given a coideal $\mathcal{H}$ of $\finkinf$ and $A\in\mathcal{H}$, for every $(\mathcal{D}_a)_{a\in\finkfin\upharpoonright A}$ such that each $\mathcal{D}_a$ is dense open in $\mathcal H\cap[\depth_A(a),A]$ there exists $(A_n)_{n\in\mathbb N}\subseteq\mathcal{H}\upharpoonright  \,A$ such that $A_n\in\mathcal{D}_a$ for all $a\in\finkfin\upharpoonright A$ with $\depth_A(a)=n$.  
\end{lem}

\begin{proof}
For every $n\in\mathbb{N}$, list 
$$\{a_1,a_2,\dots,a_{k_n}\}=\{a\in\finkfin\upharpoonright A\colon \depth_A(a)=n\}$$
Since each $\mathcal{D}_a$ is dense open in $\mathcal H\cap[\depth_A(a),A]$, for each $i=1, \dots, k_n$, $D_{a_i}$ is dense in $[n,A]$.
So, we can choose $A^1\in D_{a_1}\upharpoonright A$, $A^2\in D_{a_2}\upharpoonright A^1$, $\dots A^{k_n}\in D_{a_{k_n}}\upharpoonright A^{k_n-1}$.
Clearly, $A^{k_n} \in \bigcap_{i=1}^{k_n}D_{a_i}$. Put $A_n=A^{k_n}$; then $(A_n)_{n\in\mathbb N}$ is the desired sequence.
We have  that for every $n$, $A_n\in \mathcal{H}\upharpoonright A$, and $A_n\in D_a$ for every $a\in \finkfin\upharpoonright A$ with $\depth_A(a)=n$.
\end{proof}

% REVISAR!!!!!

\begin{prop}\label{semiequiv}
 A coideal $\cH\subseteq\finkinf$ is \textbf{semiselective} if for every $A\in\cH$ and every set $\cD=\{D_n:n\in\omega\}$, each $D_n$ dense open subset of 
 $\cH\cap [n, A]$, there exists $B\in\cH\upharpoonright A$ 
 %a \textbf{diagonalization of $\cD$ within $A$}, 
 such that  for each $s\in\finkfin\upharpoonright B$, $B/s\in D_{\depth_A(s)}$.
\end{prop}

\begin{proof}
Given $A\in  \cH$, and  $(D_n)_{n\in \mathbb N}$, we define $D_a=D_n$ for every $a\in \finkfin\upharpoonright A$ with $\depth_A(a)=n$. A diagonalization of $(D_a)$ diagonalizes $(D_n)$ within $A$.

Conversely, given $A\in \cH$, and $\mathcal D = \{\mathcal D_a\}_{a\in \finkfin\upharpoonright A}$ such that each $\mathcal D_a$ is dense open in $\mathcal H\cap [\depth_A(a),A]$, 
put for every $n\in \mathbb N$, $D_n=\bigcap_{\depth_A(a)=n}D_a$.
For every    $C\in\mathcal{H}\upharpoonright  A$, there is $B\in \cH\upharpoonright  C$ such that  $B/b\in D_{n}$ for every $b\in \finkfin\upharpoonright B$ with $\depth_A(b)=n$. So, defining $A_a$ so that $r_{|a|}(A_a) = a$ and $A_a/a = B/a$ if $a\in  \finkfin\upharpoonright B$, and choosing any $A_a\in\mathcal D_a$ if $a\notin  \finkfin\upharpoonright B$, the family $\mathcal A=\{A_a\}_{a\in \finkfin\upharpoonright A}$ verifies that $B$ is a diagonalization of $(D_a)$.
\end{proof}

%In the case that the coideal $\mathcal H$ is an ultrafilter, proposition \ref{equivselective} can be proved directily without using lemma \ref{filtered}. 

\bigskip
If $A,B\in \finkinf$, we write $A\leq^* B$ to express that $A$ is almost a condensation of $B$, that is, except for a finite number of its elements, 
every element of $A$ belongs to  $[B]$.

\begin{lem}\label{distributive}
A coideal $\cH$ is semiselective iff it is $\sigma$-distributive with respect to $\leq^\ast$.
\end{lem}

\begin{proof} 
  %  \begin{enumerate}
      Let $\mathcal H$ be a semiselective  coideal. Let $\cD=\{D_n:n\in\omega\}$ be a family of $\leq^\ast$-dense open subsets of $\cH$ and let $W_A:=\{B\in\cH:B\leq^\ast A\}$. Notice that each  $D_n$ is dense open in $\mathcal H$. Clearly, a diagonalization   $B\in\cH$  of the $D_n$'s within $A$ is an element of both $\bigcap\cD$ and $W_A$.

        Conversely, given $A\in\cH$ and $\cD$ as in Proposition \ref{semiequiv},  let  $B\in\cH$  such that $B\in\bigcap\cD$ and $B\leq^\ast A$. We can assume $B\in\cH\upharpoonright A$ without a loss of generality. Then, for any $s\in\finkfin\upharpoonright  B$ with $\depth_A(s)=n$, it holds that $B/s\in D_n$ since $B\in D_n$ and $B/s \leq^\ast B$. 
%    \end{enumerate}
\end{proof}

The next theorem si also a consequence of Lemma \ref{filtered}.
\begin{thm}\label{selective is semi} If $\mathcal H\subseteq\finkinf$ is a selective coideal then $\mathcal H$ is semiselective.
\end{thm}

\begin{proof} Consider $A\in\mathcal H$ and let $\mathcal{D}=(\mathcal D_a)_{a\in \finkfin\upharpoonright A}$ be such that each $\mathcal{D}_a$ is dense open in $\mathcal H\cap[\depth_A(a),A]$. Fix $B\in\mathcal H\cap[a,A]$. Then  $\mathcal{D}_a$ is dense open in $\mathcal H\cap[a,B]$, for all $a\in \finkfin\upharpoonright B$. Using lemma \ref{filtered} we can build $\mathcal{A}=(A_n)_{n\in\mathbb N}\subseteq\mathcal{H}\upharpoonright B$ such that $A_n\in\mathcal{D}_a$ for every $a\in\finkfin\upharpoonright B$ with $\depth_A(a)=n$. By selectivity, there exists $C\in\mathcal H\cap[a,B]$ which diagonalizes $\mathcal{A}$. Hence, $\mathcal H$ is semiselective.
\end{proof}

\begin{defn}\label{ultra}
 A family  $\mathcal U\subseteq \finkinf$ is an \textbf{ultrafilter} if it satisfies the following:
\begin{itemize}
\item[{(a)}] $\mathcal U$ is a \textit{filter} on $( \finkinf,\leq)$ invariant under finite changes. That is:  

\begin{enumerate}
\item If $A\in \mathcal U$ and $A\triangle B$ is finite, then $B\in \mathcal U$.
\item For all $A,B\in\finkinf$, if $A\in\mathcal U$ and $A\leq B$ then $B\in\mathcal U$.
\item For all $A,B\in\mathcal U$, there exists $C\in\mathcal U$ such that $C\leq A$ and $C\leq B$.
\end{enumerate}

\item[{(b)}] If $\mathcal U'\subseteq\finkinf$ is a filter on $(\finkinf,\leq)$ and $\mathcal U\subseteq \mathcal U'$ then $\mathcal U'=\mathcal U$. That is, $\mathcal U$ is a \textit{maximal filter} on $(\finkinf,\leq)$.

\item[{(c)}]  (${\bf A3}\mod\mathcal U$) For all $A\in\mathcal U$ and $a\in\finkfin\upharpoonright  A$ with $\operatorname{depth}_{A}(a)=n$, the following holds: 

\begin{itemize}
		\item[(c.1)] $(\forall B\in [n,A]\cap\mathcal U)\ \ ([a,B]\cap\mathcal U\neq\emptyset)$.
		\item[(c.2)] $(\forall B\in [a,A]\cap\mathcal U)\ \ (\exists A'\in [n,A]\cap U)\ \
		([a,A']\subseteq [a,B])$.
\end{itemize}

\item[{(d)}]  (${\bf A4}\mod\mathcal U$) Let $A\in\mathcal U$ and $a\in\finkfin\upharpoonright  A$ be given. For all $\mathcal O \subseteq \FIN_k^{|a|+1}$ then there exists $B\in [\depth_A(a),A]\cap\mathcal U$ such that $r_{|a|+1}[a,B]\subseteq\mathcal O$ or $r_{|a|+1}[a,B]\cap\mathcal O=\emptyset$.

\end{itemize}

\end{defn}

An ultrafilter is in particular a coideal, so using definifions \ref{semiselective} and \ref{selective} we can consider semiselective and selective ultrafilters on $\finkinf$.

%For ultrafilters, selectivity can be also expressed by diagonalizations of sequences indexed by $\finkfin$ as shown by the next proposition.
%Question: for ultrafilters on $\fininf$,  is semiselectivity equivalent to selectivity? (in $\ninf$ this is true because semiselective ultrafilters are Ramsey, and  an ultrafilter is Ramsey iff it is selective).

%\begin{prop}\label{equivselective}
%An ultrafilter  $\mathcal U$ is selective if and only if for every  %$[a,A]\neq\emptyset$ with 
%$A\in \mathcal U$, for every sequence $\mathcal{A}=(A_a)_{a\in \finkfin\upharpoonright \, A}\subseteq\mathcal{U}\upharpoonright A$  there exists $B\in\mathcal U \upharpoonright  A$ such that  for every $b\in\finkfin\upharpoonright B$ we have $[b, B]\subseteq [b,A_b]$. In this case we say that $B$ diagonalizes $\mathcal A$ within $A$. 
%\end{prop}

%\begin{proof}
 %\end{proof}

Will show below, in section \ref{selandstab}, that for ultrafilters on $\finkinf$ semiselectivity is equivalent to selectivity. This is also the case for ultrafilters on the space $\ninf$ as was shown by Farah in \cite{Fa}.

\section{Ramsey subsets of $\finkinf$ }
Let $\mathcal H$ be a semiselective coideal in $\FIN_k^\infty$.
 
 \begin{lem}\label{openram}
 Let $\mathcal O\subseteq \FIN_k^\infty$ be open in the metric topology. Then, $\mathcal O$ is $\mathcal H$-Ramsey.
 \end{lem}

 We adapt the ideas of Nash-Williams, Galvin and Prikry and Farah to this context.
 Before proceeding with the proof, we give some definitions.
 
 \begin{defn}
 Fix a family $\mathcal F\subseteq \FIN_k^{<\infty}$.
 \begin{enumerate}
\item  $A\in \mathcal H$ accepts $s\in \FIN_k^{<\infty}$ if for every $B\in [s,A]\cap \mathcal H$ there is $n$ such that $r_n (B)\in \mathcal F$.
\item $A$ rejects $s$  if no $B\in [s,A]\cap \mathcal H$ accepts $s$.
\item $A$ decides $s$ if it accepts $s$ or rejects $s$.
\end{enumerate}
 \end{defn}
 
 The following facts follow from the definition.
 
 \begin{lem}\label{comb forcing}
 Fix a family $\mathcal F\subseteq \finkfin$.
  \begin{enumerate}
\item  $A\in \mathcal H$ accepts and rejects  $s\in \FIN_k^{<\infty}$ implies that $[s,A]$ is empty.
\item\label{accept heir} $A$ accepts $s$  and  $B\leq A$, $B\in \mathcal H$, then $B$ accepts $s$.
\item\label{reject heir} $A$ rejects $s$  and  $B\leq A$, $B\in \mathcal H$, then $B$ rejects $s$.
\item\label{some decide} For every $A\in \mathcal H$ and for every $s\in \FIN_k^{<\infty}$ with $s\sqsubset A$, there is $B\in[s,A]\cap \mathcal H$ that decides $s$.
\item\label{accepts longer} If $A$ accepts $s$, then it accepts every $t\in r_{|s|+1}([s,A])$.
\item\label{not accept longer} If $A$ rejects $s$, then there exists $B\in [s,A]\cap \mathcal H$ such that $A$ does not accept any $t\in  r_{|s|+1}([s,B])$.
\end{enumerate}
\end{lem}

We only prove the last fact, using that if $c: \FIN_k^n\to 2$ there is $X\in \mathcal H$ such that $c$ is constant on $r_n[X]$.
Suppose $A$ rejects $s$, and let $\mathcal O=\{t\in \FIN_k^{<\infty}: A \mbox{ accepts } t\}$. By {\bf A.4} mod $\mathcal H$, there exists $B\in\mathcal H\cap [s,A]$ such that $r_{|s|+1}[s,B]\subseteq \mathcal O$ or $r_{|s|+1}[s,B]\subseteq \mathcal O^c$. If the first alternative holds then take $C\in\mathcal H\cap [s,B]$. Let $b=r_{|s|+1}(C)$. Then $b\in\mathcal O$ and therefore $A$ accepts $b$. Since $C\in[b,A]$ then there exists $n$ such that $r_n(C)\in\mathcal F$. Therefore $B$ accepts $s$, because $C$ is arbitrary. But this contradicts that $A$ rejects $s$. Hence, $r_{|s|+1}[s,B]\subseteq \mathcal O^c$ and we are done.

 \begin{claim}
 Given $A\in\mathcal H$, there exists $D\in \mathcal H  \upharpoonright  A$ which decides every $b\in \finkfin  \upharpoonright  D$.
\end{claim}

\begin{proof} For every $a\in \finkfin\upharpoonright A$ define
$$\mathcal D_a=\{C\in\mathcal H\cap [\depth_A(a),A]\colon C {\mbox{ decides }}a\}$$
By parts \ref{accept heir}, \ref{reject heir} and \ref{some decide} of Lemma \ref{comb forcing} each $\mathcal D_a$ is dense open in $\mathcal H\cap [\depth_A(a),A]$. By semi-\linebreak selectivity, there exists $D\in \mathcal{H}\upharpoonright  A$ which diagonalizes the collection $(\mathcal D_a)_{a\in \finkfin\upharpoonright A}$. By parts \ref{accept heir} and \ref{reject heir} of Lemma \ref{comb forcing}, $D$ decides every $a\in\finkfin  \upharpoonright  D$.
\end{proof}

The following is an abstract version of the semisective Galvin lemma (see \cite{Ga,Fa}).

\begin{lem}[Semiselective Galvin's lemma for $\finkinf$]\label{galvinlocal}
Given $\mathcal F\subseteq\finkfin$, a semiselective coideal $\mathcal H\subseteq \finkinf$, and $A\in\mathcal H$, there exists $B\in \mathcal{H}\upharpoonright  A$ such that one of the following holds:
\begin{enumerate}
\item $\finkfin  \upharpoonright  B\cap\mathcal F=\emptyset$, or
\item $\forall C\in [\emptyset,B]$ $(\exists \ n\in \mathbb{N})$ $(r_n(C)\in\mathcal F)$.
\end{enumerate}
\end{lem}

\begin{proof} Consider $D$ as in the Claim. If $D$ accepts $\emptyset$ part (2) holds and we are done. So assume that $D$ rejects $\emptyset$ and for $a\in \finkfin  \upharpoonright  D$ define
$$\mathcal D_a=\{C\in\mathcal H\cap [\depth_A(a),D]\colon C {\mbox{ rejects every }}b\in r_{|a|+1}([a,C])\}$$if $D$ rejects $a$, and $\mathcal D_a=\mathcal H\cap [\depth_A(a),D]$, otherwise. By parts \ref{reject heir} and \ref{not accept longer} of Lemma \ref{comb forcing} each $\mathcal D_a$ is dense open in $\mathcal H\cap [\depth_A(a),D]$. By semiselectivity, choose $B\in \mathcal{H}\upharpoonright  D$ such that for all $a\in\finkfin\upharpoonright B$ there exists $C_a\in\mathcal D_a$ with $[a,B]\subseteq [a,C_a]$. For every $a\in \finkfin\upharpoonright B$, $C_a$ rejects all $b\in r_{|a|+1}([a,C_a])$. So $B$ rejects all $b\in r_{|a|+1}([a,B])$: given one such $b$, choose any $\hat{B}\in\mathcal H\cap [b,B]$. Then $\hat{B}\in\mathcal H\cap [b,C_a]$. Therefore, since $C_a$ rejects $b$, $\hat{B}$ does not accept $b$.

Hence, $B$ satisfies that $\finkfin\upharpoonright B\cap\mathcal F=\emptyset$. This completes the proof of the Lemma.
\end{proof}

\begin{notation}
$\finkfin\upharpoonright [a,B] = \{b\in\finkfin : a\sqsubseteq b\ \&\ (\exists n\geq |a|)(\exists C\in [a,B])\ b=r_n(C)\}$.

\end{notation}

In a similar way we can prove the following generalization of lemma \ref{galvinlocal}: 

\begin{lem}\label{galvinlocal2}
Given a semiselective coideal $\mathcal H\subseteq \finkinf$, $\mathcal F\subseteq\finkfin$,  $A\in\mathcal H$ and $a\in\finkfin\upharpoonright A$, there exists $B\in \mathcal H\cap [a,A]$ such that one of the following holds:
\begin{enumerate}
\item $\finkfin\upharpoonright [a,B]\cap\mathcal F=\emptyset$, or
\item $\forall C\in [a,B]$ $(\exists \ n\in \mathbb{N})$ $(r_n(C)\in\mathcal F)$.
\end{enumerate}
\end{lem}

Now, we proceed to prove Lemma \ref{openram}. 
Recall that the basic metric open subsets of $\finkinf$ are of the form $[b]=\{A\in\finkinf\colon b\sqsubset A\},$
where $b\sqsubset A$ means  $(\exists n\in\mathbb{N})\ (r_n(A)=b)$.

  \begin{proof} (of lemma \ref{openram})
  Let $\mathcal{X}$ be a metric open subset of $\mathcal{R}$ and fix a nonempty $[a,A]$ with $A\in\mathcal{H}$. Without a loss of generality, we can assume $a = \emptyset$. Since $\mathcal{X}$ is open, there exists $\mathcal{F}\subseteq  \finkfin $ such that $\mathcal X = \bigcup_{b\in \mathcal F} [b]$. Let $B\in \mathcal{H}\upharpoonright  A$ be as in Lemma \ref{galvinlocal}. If (1) holds then $[0,B]\subseteq\mathcal{X}^c$ and if (2) holds then $[0,B]\subseteq\mathcal{X}$.
 \end{proof}

 \section{Ultrafilters and forcing}
 \subsection{Ultrafilters on $\mathbb N$ and on $FIN$}
We recall some definitions related to ultrafilters on $\omega$.

A non-principal ultrafilter $\mathcal U$ on $\mathbb N$ is a {\bf P-point} if for every partition $\mathbb N=\bigcup_{i\in \omega}A_i$ into sets not belonging to $\mathcal U$ there is $B\in \mathcal U$ such that $|B\cap A_i|<\omega$ for every $i\in \omega$. $\mathcal U$ is said to be a {\bf Q-point} if for every partition $\mathbb N=\bigcup_{i\in \omega}A_i$ into finite sets  there is $B\in \mathcal U$ such that $|B\cap A_i|\leq 1$ for every $i\in \omega$. An ultrafilter $\mathcal U$ is {\bf selective} if it is a P-point and a Q-point.

 \begin{defn}
An ultrafilter  $\mathcal U$ on $\mathbb N$is {\bf strongly summable} if for every $A\in \mathcal U$ there is a strictly increasing sequence of positive integers $\{n_k: k\in \omega\}$ such that $FS(\{n_k\})$ is an element of $\mathcal U$ and
 $FS(\{n_k\})\subseteq A$.
 \end{defn}
 
 Here, $FS(\{n_k\})= \{\sum_{k\in F}n_k: F\in [\omega]^{<\infty}\}.$ That is, $FS(\{n_k\})$ is the set of finite sums of elements of $\{n_k\}$ with no repetitions.

 Strongly summable ultrafilters are related to Hindman's theorem as selective ultrafilters are related to Ramsey's theorem.

 For ultrafilters on $FIN$, the following definitions due to Blass (\cite{Bl}) give the corresponding ultrafilters  related to the finite unions version of Hindman's theorem.
 
 \begin{defn}(\cite{Bl})\label{uultra}
 An ultrafilter $\mathcal U$ on $FIN$ is a union ultrafilter if it has a basis of sets of the form $FU(\{a_n: n\in \omega\})$ where 
   $\{a_n: n\in \omega\}$ is a sequence of pairwise disjoint elements of $FIN$.
   
   $\mathcal U$ is an ordered-union ultrafilter if it has a basis of sets of the form $FU(\{a_n: n\in \omega\})$ where 
   $\{a_n: n\in \omega\}$ is a block sequence of pairwise disjoint elements of $FIN$ (i.e. for every $n$ $max(a_n)<min(a_{n+1}$)).
 \end{defn}

Here,  $FU(\{a_n:n\in \omega\})= \{\cup_{k\in F}a_k: F\in [\omega]^{<\infty}\}.$

 \bigskip
 Recall the following fact about forcing with  coideals on $\ninf$.  If $\mathcal H$ is a selective coideal on $\ninf$, forcing with the partial order  $(\mathcal H, \subseteq^*)$ adds a selective ultrafilter  on $\mathbb N$ contained in $\mathcal H$ (\cite{Ma}).  Farah,  in  \cite{Fa}, proved the same  for $\mathcal H$ semiselective.

 \begin{lem}\label{ul}(\cite{Ma}, \cite{Fa})
 Let $M$ be a transitive model of $ZF+DC$, and $\mathcal H$ a semiselective coideal  contained in  $\ninf$ in $M$.
 Let $\mathcal U$ be $(\mathcal H, \subseteq^*)-$generic over $M$. Then,   in $M[\mathcal U]$, $\mathcal U$ is a selective ultrafilter on $\omega$.
\end{lem}

If $A,B\in \fininf$, $A\leq^* B$ means that $A$ is almost a condensation of $B$, that is, except for a finite number of its elements, every element of $A$ is a finite union of elements of $B$.

If $\mathcal H$ is a semiselective coideal on $\fininf$, forcing with $(\mathcal H, \leq^*)$  adds an ultrafilter contained in $\mathcal H$ with interesting properties.

\begin{defn}(\cite{Bl})\label{souultra}
 An ordered-union ultrafilter $\mathcal U$ on $FIN$  is stable if for every sequence $\{D_n: n\in \omega\}\subseteq \fininf$
 such that $FU(D_n)\in \mathcal U$ for every $n$, there is $E\in \fininf$ such that $FU(E)\in \mathcal U$ and for every $n$ $E\leq^*D_n$.
  \end{defn}

As we will see, if $\mathcal H$ is a semiselective coideal, then the partial order $(\mathcal H, \leq^*)$ adds a stable ordered-union ultrafilter.
This result is due to Eisworth, who  proved it in \cite{Ei02}  using the concept of  Matet-adequate families $\mathcal H\subseteq \fininf$. 

\begin{defn}\label{matetad}\cite{Ei02}
A family $\mathcal H\subseteq \fininf$ is Matet-adequate if
\begin{enumerate}
\item $\mathcal H$ is closed under finite changes,
\item  For all $A,B\in\fininf$, if $A\in\mathcal H$ and $A\leq B$ then $B\in\mathcal H$.
\item $(\mathcal H, \leq^*)$ is $\sigma$-closed,
\item If $A\in \mathcal H$ and $FU(A)$ is partitioned into 2 pieces then there is $B\leq A$ in $\mathcal H$ so that $FU(B)$ is included in a single piece of the partition (this is called the Hindman property).
\end{enumerate}
\end{defn}

  Semiselective coideals and Matet-adequate subfamilies of $\fininf$ are clearly related.  Both classes of families share the two properties of being closed under finite changes and being closed upwards; they share also the Hindman property, which is equivalent to  $A4$ (mod $\mathcal H$).
A Matet adequate family is $\sigma$-closed as a partial order under $\leq^*$, and thus $\sigma$-distributive; while a semiselective coideal $\mathcal H$,  has the diagonalization property  of definition \ref{semiselective} which also implies that $(\mathcal H, \leq^*)$ is $\sigma$-distributive.
We will prove that in fact Matet-adequate families are just selective coideals in $\fininf$.
\subsection{Ultrafilters on $\finkinf$}
We will work now in the context of the space $\finkinf$, with generalizations of the mentioned concepts to this context.  So, we fix $k\geq 1$ from now on.

The poset $(\finkk,\leq)$ (and therefore $(\finkk,\leq^\ast)$) has a maximal element $\one$ consisting of  the set of functions $k\cdot\chi_{\{n\}}$ varying $n\in\omega$. 

\begin{defn}

    \begin{enumerate}
        \item An ultrafilter $\mathcal U$ on $\fink$ is  \textbf{ordered-T} if it has a basis of sets of the form  $[A]$ with $A\in\finkinf$.
        \item An ordered-T ultrafilter is  \textbf{stable} if for every sequence $\{[A_n]:n\in\omega\}\subseteq\mathcal U$, with each $A_n\in \finkinf$,  there exists $E\in\finkinf$ such that $E\leq^\ast A_n$ for each $n\in\omega$ and $[E]\in \mathcal U$.
  
\item An ultrafilter  $\mathcal U$ on $\FIN_k$ has the \textbf{Ramsey property for pairs} if for every partition $$c:\fink^{[2]}\to 2$$ 
 there exists $X\in\mathcal U$ such that $[X]^2_<$ is monochromatic.
 \end{enumerate}
 \end{defn}
  These notions coincide with those  of ordered-union, stable, and Ramsey property for pairs as in \cite{Bl} when $k=1$.
%For a partition $c:[\fink]^2_<\to 2$ we use the notation $c(f<g):=c(\{f,g\})$ when $f<g$.

 \begin{thm}(\cite{Ei02})
 Let $\mathcal H$ be a  $\sigma$-distributive  coideal on $\finkinf$.
 If $G$ is $(\mathcal H, \leq^*)$-generic over $V$, then 
 $$\mathcal U_G=\{A\subseteq \FIN_k: \exists B\in G \,([B]\subseteq A)\}$$ is a stable ordered-T ultrafilter.

 \end{thm}

\begin{proof}
Clearly, $\mathcal U_G$ is an ordered-T filter.
The stability follows from the $\sigma$-distributivity as follows. Let $\{A_n:n\in \omega\}$ be  a sequence of elements of $\mathcal U_G$.
Take a sequence $\{B_n: n\in \omega\}$ of elements of $G$ such that for every $n\in \omega$ $[B_n]\subseteq A_n$.
For every $n$, let $$D_n=\{B\in \fininf : B\perp B_n \mbox{ or } B\leq^*B_n\}.$$
Each $D_n$ is dense in $(\mathcal H, \leq^*)$,  and so, by $\sigma$-distributivity of $\mathcal H$ the intersection 
$\bigcap_n D_n$  is dense and is in the ground model. Thus there exists $E\in G\cap (\bigcap_n D_n)$. Then $[E]\in U_G$ and $E$ is an almost condensation  of $B_n$ for every $n$.
For every partition of $\FIN_k$ in the ground model,  Gowers property $A4$ of $\mathcal H$ gives $A\in G$ such that $[A]$ is included in one part of the partition. But since $\mathcal H$ is $\sigma$-distributive, there are no new partitions of $\FIN_k$ in the extension and thus $\mathcal U_G$ is an ultrafilter.
\end{proof}

\begin{defn}
If $\mathcal U$ is an ordered-T ultrafilter on $\FIN_k$, then 
$$\mathcal U^{\infty} =\{A\in FIN^\infty : [A]\in \mathcal U \}.$$
\end{defn}

The next definition was first proposed by Krautzberger in \cite{Kra} and will be used in the proof  of Theorem \ref{stableRam}.

\begin{defn}
Let $A\in\finkinf$ and $f\in [A]$. We say that $n\in\omega$ is an \textbf{$A$-splitting point of $f$} if both $f\upharpoonright(n+1)$ and $f\upharpoonright\omega\setminus(n+1)$ are elements of $[A]$ and there exists $g\in[A]$ such that
$$f\upharpoonright (n+1)<g<f\upharpoonright \omega\setminus(n+1).$$
We define $\pi_A:[A]\to\omega$ as $\pi_A(f):=|\{n\in\omega:n\text{ is an }A\text{-splitting point of }f\}|$.
\end{defn}

\begin{thm}\label{stableRam}(\cite{Bl})
For any ordered-T ultrafilter ultrafilter $\mathcal U$ on $\FIN_k$, the following  are equivalent. 
\begin{enumerate}
\item $\mathcal U$ is stable.
\item $\mathcal U^\infty$ is selective
\item $\mathcal U$ has the Ramsey property for pairs: Whenever $[\FIN_k]^2_<$ is partitioned into two pieces, there is $X\in \mathcal U$ with $[X]^2_<$ included in one piece of the partition.
(Here $[X]^2_<$ is the set of pairs $(s,t)$of elements of $X$ with $s<t$).
\item $\mathcal U$ has the Ramsey property: For every $n\in \omega$, if $[\FIN_k]^n_<$ is partitioned into two pieces, there is an $X\in \mathcal U$ with $[X]^n_<$ included in one of the pieces.
\end{enumerate}
\end{thm}

\begin{proof}
We prove the equivalence between the first three properties following ideas of  Blass, Eisworth and Krautzberger (for the case $k=1$) in \cite{Bl}, \cite{Ei02} and \cite{Kra} respectively.

    \begin{enumerate}
        \item[$1\to2$]
        
        Suppose $\mathcal U$ is a stable ordered-T ultrafilter on $\FIN_k$. To see that $\mathcal U^\infty $ is invariant under finite changes it is enough to notice that for $A\in \mathcal U^\infty $, and $s$ is an element of $A$, the set $\{t\in [A]: \text{supp}(s)\subseteq \text{supp}(t)\}$ is not in $\mathcal U$. This is so because this set does not contain any set of the form $[X]$ with $X\in \fininf$. Once we have this, is is clear that the same applies to $\{t\in [A]: \exists s\in F (\text{supp}(s)\subseteq \text{supp}(t)\}$ for any finite subsequence $F$ of $A$.
Finally, if $B$ differs from $A$ in a finite set, then $[B]$ contains $[C]$ for some $C$ obtained deleting finitely many elements from $A$.

 It is easy to verify that $\mathcal U^\infty $ is closed upwards, since if $A\leq B$, $[A]\subseteq [B]$.
 
 If $A,B\in \mathcal U^\infty$, then $[A]$ and $[B]$ are in $\mathcal U$ and so their intersection is also in $\mathcal U$. Then, there is $C\in \mathcal U^\infty$ with $[C]$ contained in this intersection, and therefore $C\leq A, B$. We have then that $\mathcal U^\infty$ is a filter. 
 %This will be used in the second part of this proof.
 That $\mathcal U^\infty$ is a maximal filter follows from the maximality of $\mathcal U$ as an ultrafilter on $\FIN_k$.
A3 (mod $\mathcal U^\infty$) is obvious. And A4 (mod $\mathcal U^\infty$) is just a consequence of the fact that stable ordered-T ultrafilters contain homogeneous sets for partitions as in Gowers' theorem.

        We show now that $\mathcal U^\infty$ is stable.
         Let $A\in\mathcal U^\infty$ and $\{A_n:n\in\omega\}\subseteq\mathcal U^\infty\upharpoonright A$ as in definition \ref{selective}. As $\mathcal U$ is stable let $E\in\mathcal U^\infty\upharpoonright A$ such that $E\leq^\ast A_n$ and let $j:\omega\to\omega$ such that $E/j(n)\leq A_n$ and $j(n)$ is minimal with such a property (this guarantees that the function $j$ is increasing).
        
        First of all let us see that $\{f\in[E]:j(\min(f))<\max(f)\}\in\mathcal U$. Note that given $B\in\mathcal U^\infty\upharpoonright E$, there exists $f<g$ both in $[B]$ such that $j(\min(f))<\max(g)$ and therefore
        $$j(\min(f+g))=j(\min(f))<\max(g)=\max(f+g)$$
        that is, the set above intersect any element of $\mathcal U$. Let $C\in\mathcal U^\infty\upharpoonright E$ such that $[C]$ is contained in that set.
        
        Let us see now that $\{f\in[C]:\pi_C(f)\text{ is odd}\}\in\mathcal U$. For this note that any condensation of $C$ contains $f<g<h$ and $\pi_C(f+h)=\pi_C(f)+\pi_C(h)+1$ so at least one between $f,h$ and $f+h$ has an odd number of $C$-splitting points. Take now $B\in\mathcal U^\infty\upharpoonright C$ such that $[B]$ is contained in the set above and lets see that $B$ is a diagonalization of the $A_n$'s within $A$. Given $f\in[B]$ and $h\in B/f$ we have that there exists $g\in[C]$ with $f<g<h$, otherwise $\pi_C(f+h)=\pi_C(f)+\pi_C(h)$, that is an even number. We have then that
        $$\min(h)>\max(g)>j(\min(g))\geq j(\max(f))$$
        and therefore $h\in A_{\max(f)}\leq A_{\depth_A(f)}$ since $\depth_A(f)\leq\max(f)$.
        
        \item[$2\to 3$] Let $c:[\fink]_<^2\to2$ a partition and for each $f\in\fink$ consider the induced partition $c_f:\fink/f\to 2$ given by $c_f(g)=c(f, g)$. As $\mathcal U$ has the Gowers property (for being ordered-T) let $A_f\in\mathcal U^\infty$ such that $[A_f]$ is monochromatic for $c_f$. By a counting argument there exists an infinite $F\subseteq\fink$ and $i\in2$ such that for all $f\in F$, $c_f``[A_f]=\textbf{i}$. Let $G_n:=\{f\in F:\max(f)=n\}$. Take $A_0=\one$ and, since $G_n$ is finite, let $A_{n+1}\in\mathcal U^\infty\upharpoonright A_n$ such that $A_{n+1}\leq A_f$ for each $f\in G_{n+1}$. By selectivity let $B\in\mathcal U^\infty$ be a diagonalization of the $A_n$'s within $\one$. We have then that, for $f\in[B]$,
        $$B/f\leq A_{\depth_\one(f)}=A_{\max(f)}\leq A_f$$
        and therefore, if $f< g$ are both in $[B]$, $c(f, g)=c_f(g)=i$.
        
        \item[$3\to 1$] Let $\{A_n:n\in\omega\}\subseteq\mathcal U^\infty$ and consider the partition of $[\fink]_<^2$ in two pieces given by
        $$c(f, g)=\left\{
        \begin{array}{cc}
            1 & (\forall n\leq\max(f))(g\in A_n)\\
            0 & \text{otherwise}
        \end{array}
        \right.$$
        By the Ramsey property for pairs let $B\in\mathcal U^\infty$ such that $[[B]]_\less^2$ is monochromatic and note that such a set must meet the color $1$. To see this let $f\in[B]$ and choose some $f < g$ with
        $$g\in[B]\cap\bigcap\left\{[A_n]:n\leq\max(f)\right\}\in\mathcal U.$$
        Now fix $n\in\omega$ and take $f\in B$ with $n\leq\max(f)$. We have then that $B/f\leq A_n$ thus $B\leq^\ast A_n$.
    \end{enumerate}
\end{proof}

We go back now to the case $k=1$ to show that Matet-adequate families are the same as selective coideals.

%With this argument, which is inspired on a similar one used in \cite{Ei02}, we have proved the following.

\begin{thm}
Every Matet-adequate family is a selective coideal in $\fininf$. Therefore,  Matet-adequate families and selective coideals coincide.
\end{thm}

\begin{proof} (inspired on \cite{Ei02})
Let  $\mathcal H$ be a Matet-adequate family, and let $[a,A]\neq \emptyset$ with $A\in \mathcal H$, and $\{A_n: n\in \omega\}$ a $\leq$-decreasing sequence in $\mathcal H$, with each $A_n$ a condensation of $A$ such that $[a,A_n]\neq \emptyset$. Let $C\in \mathcal H$  be such that $C\leq^* A_n$ for every $n\in \omega$.
Force with $(\mathcal{H}\upharpoonright C, \leq^*)$, the collection of all condensations of $C$ which are in $\mathcal H$. This adds a stable ordered-union ultrafilter on $FIN$ such that 
$FU(C)\in \mathcal U$. In fact, if $G$ is a generic subset of $\mathcal{H}\upharpoonright C$, then $\mathcal U_G=\{A\subseteq FIN: \exists B\in G \,(FU(B)\subseteq A)\}$ is a stable union-ordered ultrafilter and $FU(C)$ belongs to it.

Notice that for every $n$, $FU(A_n)\in \mathcal U$. Applying Theorem \ref{stableRam} %and Corollary \ref{stablesel} 
we get $B\leq C$ with $FU(B)\in \mathcal U$,  which diagonalizes 
$\{A_n: n\in \omega\}$ whitin $A$.

$B$ is in the ground model since no reals are added by this forcing. 
And $B\in \mathcal H$ by definition of $\mathcal U_G$ and the fact that $\mathcal Hose$ is upwards closed.
Using the fact that every selective coideal is $\sigma$-closed the proof is finished.
\end{proof}

\section{Selective ultrafilters and stable ordered-T ultrafilters}\label{selandstab}

In this section we complete our remarks concerning  the relation between stable ordered-union ultrafilters on $FIN$ and selective ultrafilters on the topological Ramsey space $\fininf$.

Recall that for an ultrafilter $\mathcal U$ on $\FIN_k$, 
$$\mathcal U^\infty=\{A\in \finkinf : [A]\in \mathcal U\}.$$

By Theorem \ref{stableRam}, 
if $\mathcal U$ is a stable ordered-T ultrafilter on $\FIN_k$, then $\mathcal U^\infty$ is a selective  ultrafilter on $\finkinf$. The following is a sort of reverse implication.

\begin{thm}\label{uv}
 If $\mathcal V$ is a selective  ultrafilter on $\finkinf$, then the filter $\mathcal U$ on $\FIN_k$ generated by $\{ [A]: A\in \mathcal V\}$ is a stable ordered-T ultrafilter. 
\end{thm}

\begin{proof}

Let $\mathcal V$ be a selective utrafilter on $\finkinf$, and let $\mathcal U$ be the filter on $\FIN_k$ generated by $\{[A]: A\in \mathcal V\}$. $\mathcal U$ is in fact a filter because if $A,B\in \mathcal V$ then there exists $C\in \mathcal V$ such that $C\leq A$ and $C\leq  B$, and it is obvious that  it is an ordered-T filter.
To see that it is an ultrafilter on $\FIN_k$, suppose that there is a ordered-T ultrafilter $\mathcal U'$ properly extending $\mathcal U$. Take $X\in \mathcal U'\setminus \mathcal U$.
Then there is $A\in \finkinf$ such that $[A]\subseteq X$, and $A\notin \mathcal V$. But then $\mathcal (U')^\infty$ is a filter properly containing $\mathcal V$.

Let us now verify that $\mathcal U$ is stable. Given a sequence $\{A_n: n\in \omega\}\subseteq \finkinf$
 such that $[A_n]\in \mathcal U$ for every $n$, we use the selectivity of $\mathcal V$ to obtain $E$ which is an almost condensation of each $A_n$. 
 For every $n$, since $[A_n]\in \mathcal U$, there is $B_n\in \mathcal V$ such that $[B_n]\subseteq [A_n]$, and therefore $B_n\leq A_n$.
 Now we work with the $B_n$'s, and using the fact that they belong to the ultrafilter $\mathcal V$, we construct a descending sequence $\{C_n: n\in \omega\}$
 as follows, $C_0=B_0$ and if $C_n$ has been defined, we let  $C_{n+1}$ be an element $C$ of $\mathcal V$ such that $C\leq C_n $ and $C\leq B_n$.
 The selectivity of $\mathcal V$ gives us a diagonalization $E$ of the sequence $\{C_n: n\in \omega\}$ in the ultrafilter. Then $E$ is an almost condensation of each $A_n$.
 
\end{proof}

We will now prove that any semiselective ultrafilter on $\finkinf$ is in fact selective.

\begin{thm}
If $\mathcal V$ is a semiselective ultrafilter on $\finkinf$, then it is selective.
\end{thm}

\begin{proof}
Notice that the ultrafilter $\mathcal U$ on $\FIN_k$ generated by $\{[A]: A\in \mathcal V\}$ is an ordered-T ultrafilter. Let us show that it has the  Ramsey property for pairs.
Let $\mathcal F\subseteq [\FIN_k]_<^2$ (thus, $\mathcal F$ determines a partition of $[FIN]_<^2$ into two parts).   Fix $A\in \mathcal V$. 
Since $\mathcal V$ is, in particular,  a semiselective coideal on $\finkinf$, and $\mathcal F\subseteq \finkfin$, by Lemma \ref{galvinlocal} there is $B\in \mathcal V\upharpoonright A$ such that either $\finfin\upharpoonright B\cap \mathcal F=\emptyset$ or $\forall C\in [\emptyset , B] (\exists n \in \mathbb N) (r_n(C)\in \mathcal F$.
In the first case, $[[B]]_<^2 \cap \mathcal F=\emptyset$; and in the second case, $[[B]]_<^2\subseteq \mathcal F$.

By Theorem \ref{stableRam} the ultrafilter $\mathcal U$ is therefore stable, 
and   the ultrafilter $\mathcal U^\infty=\mathcal V$ is selective.
\end{proof}

%%%%%%%%%%%%%%%%%%%%%%%%%%%%%%%

\section{Parametrizing with perfect sets.}

In this section we will study a parametrized version of the $\mathcal H$-Ramsey property for elements of $\finkinf$  introduced in Section \ref{block}. Let $2^{\infty}$ be the space of infinite sequences of 0's and
1's, with the product topology regarding $2 = \{0, 1\}$ as a
discrete space. Also, $2^{<\infty}$ denotes the set of finite
sequences of 0's and 1's. Let us consider some features of the
perfect subsets of $2^{\infty}$, following \cite{Paw}:

Some notation is needed. For $x = (x_{n})_n \in 2^{\infty}$, $x|_{
k}$ denotes the finite sequence \linebreak $(x_{0}, x_{1}, \cdots, x_{k-1})$.
For $u\in 2^{<\infty}$, let $[u] = \{x\in 2^{\infty}: (\exists k)
(u = x|_{k})\}$ and let $|u|$ be the lenth of $u$. Given a perfect
set $Q\subseteq 2^{\infty}$, let $T_{Q}$ be its associated perfect
tree.  For $n\in\mathbb N$, let $T_P\upharpoonright n = \{u\in T_P : |u|=n\}$. Also, for $u, v = (v_{0}, v_{1}, \cdots, v_{|v|-1})\in
2^{<\infty}$ we write $u\sqsubseteq v$ to mean $(\exists k\leq
|v|) (u = (v_{0}, v_{1}, \cdots, v_{k-1}))$. For each $u\in
2^{<\infty}$, let $Q(u) = Q\cap [u(Q)]$, where $u(Q)\in T_{Q}$ is
define inductively, as follows: $\emptyset(Q) = \emptyset$.
Suppose $u(Q)$ defined. Find $\sigma\in T_{Q}$ such that $\sigma$
is the $\sqsubseteq$-extension of $u(Q)$ where the first ramification extending $u(Q)$ occurs. Then, set $(u*i)(Q)$ = $\sigma*i$, $i = 0,1$.
Here "$*$" denotes concatenation. Note that for each $n$, $Q$ =
$\bigcup\{Q(u) : u\in 2^{n}\}$.

%Given $n\in \mathbb{N}$, and perfect sets $S$ and $Q$ we say that $S\subseteq_{n} Q$ if $S(u)\subseteq Q(u)$, for every $u\in
%2^{n}$. The relation ``$\subseteq_{n}$" is a partial order. If for every $u\in 2^{n}$ we have chosen a $S_{u}\subseteq Q(u)$, then $S
%= \bigcup_{u}S_{u}$ is perfect and we have $S(u) = S_{u}$ and
%$S\subseteq_{n} Q$. As pointed out in \cite{Paw}, the most
%important feature of this partial order is the \emph{property of
%fusion}: if $Q_{n+1}\subseteq_{n+1} Q_{n}$, $n\in \mathbb{N}$, then the fusion $Q = \bigcap_{n} Q_{n}$ is a perfect set and
%$Q\subseteq_{n} Q_{n}$, for each $n$.

\begin{defn}\label{perf ram}
Let $\mathcal H\subseteq\finkinf$ be a semiselective coideal. We say that a set $\mathcal{X}\subseteq
2^{\infty}\times\finkinf$ is \textbf{perfectly $\mathcal H$-Ramsey} if for
every perfect set $Q\subseteq 2^{\infty}$ and every neighborhood
$[a,A]\neq\emptyset$ in $\finkinf$ with $A\in\mathcal H$ there exist a perfect set $S\subseteq Q$ and
$B\in [a,A]\cap\mathcal H$ such that $S\times [a,B]\subseteq \mathcal{X}$ or
$S\times [a,B]\cap \mathcal{X} = \emptyset$. A set
$\mathcal{X}\subseteq 2^{\infty}\times\finkinf$ is
\textbf{perfectly $\mathcal H$-Ramsey null} if for every perfect set
$Q\subseteq 2^{\infty}$ and every neighborhood
$[a,A]\neq\emptyset$ in $\finkinf$ with $A\in\mathcal H$ there exist a perfect set $S\subseteq Q$ and
$B\in [a,A]\cap\mathcal H$ such that $S\times [a,B]\cap \mathcal{X} =
\emptyset$.
\end{defn}

Also, we will need to adapt the notion of abstract Baire
property (see \cite{Mor}) to this context:

\begin{defn}\label{perf baire}
Let $\mathbb{P}$ be the family of perfect subsets of $2^{\infty}$ and let $\mathcal H\subseteq\finkinf$ be a semiselective coideal. We will say that a set $\mathcal{X}\subseteq 2^{\infty}\times
\finkinf$ has the $\mathbb{P}\times
Exp(\mathcal{H})$-\textbf{Baire property} if for every perfect set
$Q\subseteq 2^{\infty}$ and every neighborhood $[a,A]\neq\emptyset$ in $\finkinf$ with $A\in\mathcal H$  there exist
a perfect set $S\subseteq Q$ and a nonempty neighborhood $[b,B]\subseteq
[a,A]$ with $B\in\mathcal H$ such that $S\times [b,B]\subseteq \mathcal{X}$ or $S\times
[b,B]\cap \mathcal{X} = \emptyset$. A set $\mathcal{X}\subseteq
2^{\infty}\times\finkinf$ is $\mathbb{P}\times
Exp(\mathcal{H})$-\textbf{meager} if for every perfect set
$Q\subseteq 2^{\infty}$ and every neighborhood $[a,A]\neq\emptyset$ in $\finkinf$ with $A\in\mathcal H$  there exist
a perfect set $S\subseteq Q$ and a nonempty neighborhood $[b,B]\subseteq
[a,A]$ with $B\in\mathcal H$ such that $S\times [b,B]\cap\mathcal{X} = \emptyset$.
\end{defn}

Next, we state the main result of this section:

\begin{thm}\label{thm1}
Let $\mathbb{P}$ be the family of perfect subsets of $2^{\infty}$ and let $\mathcal H\subseteq\finkinf$ be a semiselective coideal.  The following are true:

\begin{itemize}
\item[{(a)}]$\mathcal{X}\subseteq 2^{\infty}\times\finkinf$ is perfectly $\mathcal{H}$-Ramsey iff $\mathcal{X}$ has the $\mathbb{P}\times Exp(\mathcal{H})$-Baire Property.
\item[{(b)}]$\mathcal{X}\subseteq 2^{\infty}\times \finkinf$ is perfectly $\mathcal{H}$-Ramsey null iff $\mathcal{X}$ is $\mathbb{P}\times Exp(\mathcal{H})$-meager.
\end{itemize}
\end{thm}

In order to prove Theorem \ref{thm1} we will need to introduce two combinatorial forcings and prove a series of lemmas related to them. Before doing that, we will state the following definition and a related lemma which will be useful in the sequel.

\begin{defn} Let $\mathcal S$ be a nonempty subset of $\mathbb P\times\finkinf$. A set $\mathcal D\subseteq\mathcal S$ is dense open in $\mathcal S$ if:
\begin{itemize}
\item[{(i)}] For every $(Q,B)\in\mathcal S$ there exists $(P, A)\in\mathcal D$ such that $P\times A\subseteq Q\times B$; and 
\item[{(ii)}] If $(P, A)\in\mathcal D$ and $Q\times B\subseteq P\times A$ then, $(Q, B)\in\mathcal D$.
\end{itemize}
\end{defn}

\begin{lem}\label{dense prod fusion}
Let $\mathcal H\subseteq\finkinf$ be a semiselective coideal and let $\mathcal S$ be a nonempty subset of $\mathbb P\times\mathcal H$. If $\mathcal D_n$, $n\in\mathbb N$, are dense open subsets of $\mathcal S$ then, the set \[\mathcal D_{\infty} = \{(P, A) : (\forall u\in T_P\upharpoonright n)\, (\forall a\in [A] \mbox{ with } \depth_A(a)\leq n)\; (P\cap [u], A/a)\in \mathcal D_n \}\] is dense open in $\mathcal S$. \end{lem}
\begin{proof}
The proof of Lemma \ref{dense prod fusion} follows the same argument used in the proof of Lemma 2.5 of \cite{Fa} so we will leave it to the reader.
\end{proof}
 
 \bigskip
 
 Now, fix a semiselective coideal $\mathcal H\subseteq\finkinf$.

\bigskip

{\bf Combinatorial Forcing 1.} Fix $\mathcal{F}\subseteq
2^{<\infty}\times \FIN_k^{<\infty}$. For a perfect $Q\subseteq
2^{\infty}$, $A\in \mathcal H$ and a pair $(u,a)\in
2^{<\infty}\times \FIN_k^{<\infty}$, we say that $(Q,A)$
\textit{accepts} $(u,a)$ if for every $x\in Q(u)$ and for every
$B\in [a,A]\cap\mathcal H$ there exist integers $k$ and $m$ such that
$(x|_{k},r_{m}(B))\in \mathcal{F}$. We say that $(Q,A)$
\textit{rejects} $(u,a)$ if for every perfect $S\subseteq Q(u)$
and every $B\in [a,A]\cap\mathcal H$, $(S,B)$ does not accepts
$(u,a)$. Also, we say that $(Q,A)$ \textit{decides} $(u,a)$ if it
accepts or rejects it.

\vspace{.25 cm}

{\bf Combinatorial Forcing 2.} Fix $\mathcal{X}\subseteq
2^{\infty}\times\finkinf$. For a perfect $Q\subseteq
2^{\infty}$, $A\in \mathcal H$ and a pair $(u,a)\in
2^{<\infty}\times \FIN_k^{<\infty}$, we say that $(Q,A)$
\textit{accepts} $(u,a)$ if $Q(u)\times [a,A]\subseteq
\mathcal{X}$. We say that $(Q,A)$ \textit{rejects} $(u,a)$ if for
every perfect $S\subseteq Q(u)$ and every $B\in [a,A]\cap\mathcal H$, $(S,B)$ does not accepts $(u,a)$. And as before, we say
that $(Q,A)$ \textit{decides} $(u,a)$ if it accepts or rejects it.

\vspace{.5 cm}

{\bf Note:} Lemmas \ref{lem1}, \ref{lem2}, and \ref{lem3} below
hold {\em for both} combinatorial forcings defined above.

\begin{lem}\label{lem1}
Let  a perfect $Q\subseteq
2^{\infty}$, $A\in \mathcal H$ and a pair $(u,a)\in
2^{<\infty}\times \FIN_k^{<\infty}$ be given. The following are true:
\begin{itemize}
\item[{(a)}]If $(Q,A)$ accepts (rejects) $(u,a)$ then $(S,B)$ also
accepts (rejects) $(u,a)$,  for every perfect $S\subseteq Q(u)$
and every $B\in [a,A]\cap\mathcal H$. 
\item[{(b)}]If $(Q,A)$
accepts (rejects) $(u,a)$ then $(Q,B)$ also accepts (rejects)
$(u,a)$, for every $B\in [a,A]\cap\mathcal H$. 
\item[{(c)}] There exist a perfect set $S\subseteq Q$ and  $B\in [a,A]\cap\mathcal H$ such that $(S,B)$ decides $(u,a)$. 
\item[{(d)}]If $(Q,A)$ accepts $(u,a)$ then $(Q,A)$ accepts $(u,b)$ for every
$b\in r_{|a|+1}[a,A]$. 
\item[{(e)}]If $(Q,A)$ rejects $(u,a)$ then
there exist $B\in [\depth_A(a),A]\cap\mathcal H$ such that $(Q,A)$ does not
accept $(u,b)$ for every $b\in r_{|a|+1}[a,B]$.
\item[{(f)}]$(Q,A)$ accepts (rejects) $(u,a)$ iff $(Q,A)$ accepts
(rejects) $(v,a)$, for every $v\in 2^{<\infty}$ with $u\sqsubseteq
v$.
\end{itemize}
\end{lem}

\begin{proof}

\vspace{.25 cm}

(a), (b), (c), (d) and (f) follow from the definitions. Now to
proof (e), take $(u,a)$ with $|a| = m$ and suppose $(Q,A)$ rejects
it. Define $\phi : \FIN_k^{m+1}\rightarrow 2$ such that
$\phi(b) = 1$ iff $(Q,A)$ accepts $(u,b)$. Let $n = \depth_{A}(a)$.
By A4 mod $\mathcal H$, there exists $B\in [n,A]\cap\mathcal H$ such that $\phi$ is constant on
$r_{m+1}[a,B]$.

If $\phi$ takes value 1 on $r_{m+1}[a,B]$ then $(Q,B)$ accepts $(u,a)$. So, in virtue of part (b), $\phi$ must take value 0 on $r_{m+1}[a,B]$ since $(Q,A)$ rejects $(u,a)$. Then $B$ is as required.

\end{proof}

\begin{lem}\label{lem2}
For every perfect $P\subseteq 2^{\infty}$ and $A\in \mathcal H$
there exist a perfect $Q\subseteq P$ and $B\leq A$ in $\mathcal H$ such that
$(Q,B)$ decides $(u,a)$, for every $(u,a)\in 2^{<\infty}\times
 \finkfin\upharpoonright B$ with $\depth_{B}(a)\leq |u|$.
\end{lem}
\begin{proof}

Let \[\mathcal D_n= \{(Q, B)\in\mathbb P\times\mathcal H :  Q\subseteq P, B\leq A\mbox{  and } \] \[ (\forall (u,a)\in T_Q\times  \finkfin\upharpoonright B \mbox{ with } \depth_B(b)\leq |u|=n)\; (Q, B) \mbox{ decides } (u,a)\}\] 

\medskip

Let $\mathcal S = \{(Q, B)\in\mathbb P\times\mathcal H : Q\subseteq P, B\leq A\}$. Then $\mathcal D_n$ is dense open in $\mathcal S$. Let $\mathcal D_{\infty}$ be as in  Lemma \ref{dense prod fusion}  and choose $(Q, B)\in\mathcal D_{\infty}$. Then   $Q$ and $B$ are as required.
%%%%%%%%%%%%%%%%%%%%%%%%%%%%%%%%%%%%%%%
\end{proof}

\vspace{.25 cm}

\begin{lem}\label{lem3}
Let $Q$ and $B$ be as in Lemma \ref{lem2}. Suppose $(Q,B)$ rejects
$(<>,\emptyset)$. Then there exists $D\leq B$ in $\mathcal H$ such that $(Q,D)$
rejects $(u,b)$, for every $(u,b)\in 2^{<\infty}\times
 \finkfin\upharpoonright D$ with $\depth_{D}(b)\leq |u|$.
\end{lem}
\begin{proof}

For $n\in\mathbb N$, let \[\mathcal D_n = \{C\in\mathcal H\upharpoonright B : (Q,C) \mbox{ rejects every } (u,b)\in 2^{n}\times  \finkfin\upharpoonright C\mbox{ with } \depth_{C}(b) = n\}.\]

%%%%%%%%%%%%%%%%%%%%%%%%%%%%%%%%%%

\begin{clm}
Every $\mathcal D_n$ is dense open in $\mathcal H\upharpoonright B$.
\end{clm}
\begin{proof}[Proof of Claim]
 By induction on $n\geq 1$.  Case $n=1$: Let $D\in\mathcal H\upharpoonright B$ be given. Note that $(Q, D)$ rejects  $(<>,\emptyset)$, by Lemma \ref{lem1} (b). Therefore, by parts (b), (e) and (f) of Lemma \ref{lem1}, by the choice of $(Q,B)$, and by the fact that $|b|\leq \depth_{D}(b)$, for every $b\in \finkfin\upharpoonright D$,  we can find $C_1\in\mathcal D_1$ with $C_1\leq D$. On the other hand, obviously, if $D\in\mathcal D_1$ and $C\leq D$ then, $C\in\mathcal D_1$. That is, $\mathcal D_1$ is dense open in $\mathcal H\upharpoonright B$. 
 
 Now suppose $\mathcal D_n$ is dense open in $\mathcal H\upharpoonright B$ and, again, let $D\in\mathcal H\upharpoonright B$ be given. Choose $C_n\in\mathcal D_n$ with $C_n\leq D$. Let $u_{0}$, $u_{1}$,... , $u_{2^{n+1}-1}$ be a list of the
elements of $2^{n+1}$; and let $b_{0}$, $b_{1}$,... , $b_{m}$ be a
list of the $b\in  \finkfin\upharpoonright C_{n}$ such that $\depth_{C_{n}}(b)
= n$. By Lemma \ref{lem1}(f), $(Q,C_{n})$ rejects $(u_{i},b_{j})$
for every $(i,j)\in\{0, 1,\cdots ,  2^{n+1}-1\}\times \{0,
1,\cdots , m\}$. Now, by Lemma \ref{lem1}(e) there exists
$C_{n}^{0,0}\in [n, C_{n}]\cap\mathcal H$ such that $(Q,C_{n}^{0,0})$ rejects
$(u_{0},b)$ for every $b\in r_{|b_{0}|+1}[b_{0},C_{n}^{0,0}]$.

In the same way, for every $(i,j)\in\{0, 1,\cdots ,
2^{n+1}-1\}\times \{0, 1,\cdots , m\}$,   we can find
$C_{n}^{i,j}$ satisfying the following:

\begin{enumerate} 
\item $C_{n}^{i,j+1}\in [n,C_{n}^{i,j}]\cap\mathcal H$,
\item $C_{n}^{i+1,0}\in [n,C_{n}^{i,m}]\cap\mathcal H$,  and 
\item $(Q,C_{n}^{i,j})$ rejects $(u_{i},b)$ for every $b\in
r_{|b_{j}|+1}[b_{j},C_{n}^{i,j}]$. 
\end{enumerate}

 Let $C_{n+1} = C_{n}^{2^{n+1}-1,m}$. Notice that $C_{n+1}\in \mathcal D_{n+1}$  and $C_{n+1}\leq D$. And, obviously, as in in the case $n=1$, if $D\in\mathcal D_{n+1}$ and $C\leq D$ then, $C\in\mathcal D_{n+1}$. Thus, $\mathcal D_{n+1}$ is dense open in $\mathcal H\upharpoonright B$. This completes the induction argument and the proof of the Claim.
\end{proof}

%%%%%%%%%%%%%%%%%%%%%%%%%%%%%%%%%%%%%
Then, by semiselectivity, there exists a diagonalization $D\in\mathcal H\upharpoonright B$ of the sequence $(\mathcal D_n)_n$. It is easy to see that $D$ is as required.

\end{proof}

The next theorem is inspired by Theorem 2.3 of \cite{Fa} and Theorem 3 of \cite{Mija}. 

\medskip

\begin{thm}\label{thm2}
Let $\mathcal H\subseteq\finkinf$ be a semiselective coideal. For every $\mathcal{F}\subseteq 2^{<\infty}\times \FIN_k^{<\infty}$, perfect $P\subseteq 2^{\infty}$ and $A\in\mathcal{H}$ there exist
a perfect $S\subseteq P$ and $D\leq A$ in $\mathcal H$ such that one of the
following holds:

\begin{itemize}
\item[{(a)}]for every $x\in S$ and every $C\leq D$ in $\mathcal H$ there exist
integers $l$ and $m > 0$ such that $(x|_{l},r_{m}(C))\in
\mathcal{F}$. 
\item[{(b)}]$(T_{S}\times
 \finkfin\upharpoonright D)\cap\mathcal{F} = \emptyset$.
\end{itemize}

\end{thm}
\begin{proof}
Given $\mathcal{F}\subseteq 2^{<\infty}\times \FIN_k^{<\infty}$, perfect
$P\subseteq 2^{\infty}$ and $A\in\mathcal H$, consider the
combinatorial forcing 1. Let $Q\subseteq P$ and $B\leq A$ be as in
Lemma \ref{lem2}. If $(Q,B)$ accepts $(<>,\emptyset)$ then part
(a) of Theorem \ref{thm2} holds by the definition of ``accepts". So
suppose $(Q,B)$ does not accept (and hence, rejects)
$(<>,\emptyset)$. By Lemma \ref{lem3}, find $D\leq B$ in $\mathcal H$ such that
$(Q,D)$ rejects $(u,b)$, for every $(u,b)\in
2^{<\infty}\times  \finkfin\upharpoonright D$ with $\depth_{D}(b)\leq |u|$.
Suppose towards a contradiction that there exist $(t,b)$ in
$(T_{Q}\times  \finkfin\upharpoonright D)\cap\mathcal{F}$. Find $u_{t}\in
2^{<\infty}$ such that $Q(u_{t})\subseteq Q\cap [t]$. Then $(Q,D)$
accepts $(u_{t},b)$: for $x\in Q(u_{t})$ and $C\in [b,D]$, let $l
= |t|$ and $m$ be such that $r_{m}(C) = b$. Then
$(x|_{l},r_{m}(C)) = (t,b)\in \mathcal{F}$. But then, by Lemma \ref{lem1} (f), $(Q,D)$ accepts $(v,b)$, for
every $v\in 2^{<\infty}$ such that $u_{t}\sqsubseteq v$ and
$|v|\geq \depth_{D}(b)$. This is a contradiction with the choice of
$D$. Therefore, for $S = Q$ and $D$ part (b) of Theorem \ref{thm2}
holds.
\end{proof}

\vspace{.5 cm}

Now we are ready to prove the main result of this section:

\begin{proof}[Proof of Theorem \ref{thm1}]

\vspace{.25 cm}

\textbf{(a)} The implication from left to right is obvious. So suppose $\mathcal{X}\subseteq 2^{\infty}\times \finkinf$ has the $\mathbb{P}\times Exp(\mathcal{H})$-Baire Property, and let $P\times [a,A]$ be given, with $A\in\mathcal H$.  In order to make the proof notationally simpler, we will assume $a = \emptyset$ without a loss of generality.

\vspace{.25 cm}

\begin{claim} \label{claim2}Given $\hat{\mathcal{X}}\subseteq 2^{\infty}\times
\finkinf$, perfect $\hat{P}\subseteq 2^{\infty}$ and
$\hat{A}\in\mathcal H$, there exist $Q\subseteq \hat{P}$ and
$B\leq \hat{A}$ in $\mathcal H$ such that for each $(u,b)\in
2^{<\infty}\times  \finkfin\upharpoonright B$ with $|u|\geq \depth_{B}(b)$ one
of the following holds:
\begin{itemize}
\item[i.)]$Q(u)\times [b,B]\subseteq\hat{\mathcal{X}}$
\item[ii.)]$R\times [b,C]\not\subseteq\hat{\mathcal{X}}$, for
every $R\subseteq Q(u)$ and every $C\leq B$ compatible with $b$.
\end{itemize}
\end{claim}

\begin{proof}[Proof of Claim \ref{claim2}]

\vspace{.25 cm}

Consider the Combinatorial Forcing 2 and apply Lemma \ref{lem2}.

\end{proof}

\vspace{.25 cm}

Apply the Claim \ref{claim2}  to $\mathcal{X}$, $P$ and $A$ to find
$Q_{1}\subseteq P$ and $B_{1}\leq A$ in $\mathcal H$ such that for each $(u,b)\in
2^{<\infty}\times  \finkfin\upharpoonright B_{1}$ with $|u|\geq
\depth_{B_{1}}(b)$ one of the following holds:

\begin{itemize}
\item[(1)]$Q_{1}(u)\times [b,B_{1}]\subseteq\mathcal{X}$ or
\item[(2)]$R\times [b,C]\not\subseteq\mathcal{X}$, for every
$R\subseteq Q_{1}(u)$ and every $C\leq B_{1}$ compatible with $b$.
\end{itemize}

\vspace{.25 cm}

For each $t\in T_{Q_{1}}$, choose $u_{1}^{t}\in 2^{<\infty}$ such
that $u_{1}^{t}(Q_{1})\sqsubseteq t$. \vspace{.25 cm}

Let $$\mathcal{F}_{1} = \{(t,b)\in
T_{Q_{1}}\times  \finkfin\upharpoonright B_{1} : Q_{1}(u_{1}^{t})\times
[b,B_{1}]\subseteq\mathcal{X}\}$$ Now, pick $S_{1}\subseteq Q_{1}$
and $D_{1}\leq B_{1}$   in $\mathcal H$ satisfying Theorem \ref{thm2}. If (a) of
Theorem \ref{thm2} holds then $S_{1}\times [0,D_{1}]\subseteq \mathcal{X}$ and
we are done. So suppose (b) holds. Apply  the Claim \ref{claim2} to $\mathcal{X}^{c}$, $S_{1}$ and $D_{1}$ to find
$Q_{2}\subseteq S_{1}$ and $B_{2}\leq D_{1}$  in $\mathcal H$ such that for each
$(u,b)\in 2^{<\infty}\times  \finkfin\upharpoonright B_{2}$ with $|u|\geq
\depth_{D_{2}}(b)$ one of the following holds:

\begin{itemize}
\item[(3)]$Q_{2}(u)\times [b,B_{2}]\subseteq\mathcal{X}^{c}$ or
\item[(4)]$R\times [b,C]\not\subseteq\mathcal{X}^{c}$, for every
$R\subseteq Q_{2}(u)$ and every $C\leq B_{2}$ compatible with $b$.
\end{itemize}

\vspace{.25 cm}

As before, for each $t\in T_{Q_{2}}$, choose $u_{2}^{t}\in
2^{<\infty}$ such that $u_{2}^{t}(Q_{2})\sqsubseteq t$.

\vspace{.25 cm}

Let $$\mathcal{F}_{2} = \{(t,b)\in
T_{Q_{2}}\times  \finkfin\upharpoonright B_{2} : Q_{2}(u_{2}^{t})\times
[b,B_{2}]\subseteq\mathcal{X}^{c}\}$$ Again, pick $S_{2}\subseteq
Q_{2}$ and $D_{2}\leq B_{2}$  in $\mathcal H$ satisfying Theorem \ref{thm2}. If (a)
of Theorem \ref{thm2} holds then $S_{2}\times [0,D_{2}]\cap\mathcal{X} =
\emptyset$ and we are done. So suppose (b) holds again. Let us see
that this contradicts the fact that $\mathcal{X}$ has the
$\mathbb{P}\times Exp(\mathcal{H})$-Baire Property:

Note that for every $(t,b)\in T_{S_{2}}\times  \finkfin\upharpoonright D_{2}$ the following holds:

\vspace{.25 cm}

\begin{itemize}
\item[(i)]$Q_{1}(u_{1}^{t})\times [b,B_{1}]\not\subseteq\mathcal{X}$, and
\item[(ii)]$Q_{2}(u_{2}^{t})\times [b,B_{2}]\not\subseteq\mathcal{X}^{c}$.
\end{itemize}

\vspace{.25 cm}

So, suppose there is a nonempty $R\times [b,C]\subseteq
S_{2}\times [\emptyset,D_{2}]\cap\mathcal{X}$, with  $C\in\mathcal H$, and pick $t\in
T_{R}$ with $\mid u_{1}^{t}\mid \ \ \geq \ \ \depth_{B_{1}}(b)$.
Note that $R\cap [t]\subseteq Q_{1}(u_{1}^{t})$. On the one hand
we have that $R\cap [t]\times [b,C]\subseteq R\times
[b,C]\subseteq\mathcal{X}$. But in virtue of (i),
$Q_{1}(u_{1}^{t})\times [b,B_{1}]\not\subseteq\mathcal{X}$ and
hence by (2) above we have that $R\cap [t]\times
[b,C]\not\subseteq\mathcal{X}$. If we suppose that there is a
nonempty $R\times [b,C]\subseteq S_{2}\times
[\emptyset,D_{2}]\cap\mathcal{X}^{c}$, with $C\in\mathcal H$, we reach to a similar
contradiction in virtue of (ii) and (4) above. So there is neither
$R\times [b,C]\subseteq S_{2}\times
[\emptyset,D_{2}]\cap\mathcal{X}$ nor $R\times [b,C]\subseteq
S_{2}\times [\emptyset,D_{2}]\cap\mathcal{X}^{c}$. But this is
impossible because $\mathcal{X}$ has the $\mathbb{P}\times
Exp(\mathcal{H})$-Baire Property.

\vspace{.25 cm}

\textbf{(b)} Again, the implication from left to right is obvious.
Conversely, the result follows easily from part \textbf{(a)} and the fact that
$\mathcal{X}$ is $\mathbb{P}\times Exp(\mathcal{H})$-meager. This completes the proof of Theorem \ref{thm1}.

\end{proof}

 \subsection{Closedness Under the Souslin Operation.}
 
\begin{lem}\label{RamseyNull}
Let  $\mathcal H$ be a semiselective coideal in $\finkinf$. The perfectly
$\mathcal H$-Ramsey null subsets of $2^{\infty}\times\finkinf$ form a
$\sigma$-ideal.
\end{lem}
\begin{proof}

\vspace{.25 cm}

Let $(\mathcal{X}_n)_n$ be a sequence of perfectly $\mathcal H$-Ramsey null
subsets of $2^{\infty}\times\finkinf$ and fix $P\times [a,A]$.
We can assume $a = \emptyset$. Also, it is easy to see that the finite union
of perfectly $\mathcal H$-Ramsey null sets yields a perfectly $\mathcal H$-Ramsey null set;
so we will assume $(\forall n)\ \mathcal{X}_n\subseteq
\mathcal{X}_{n+1}$. The rest of the proof is similar to the proof of Lemma \ref{lem2}. For $n\in\mathbb N$, let 

\[\mathcal D_n = \{(Q, B) : Q\subseteq P, B\leq A,\ (\forall b\in  \finkfin\upharpoonright B \mbox{ with }
\depth_{B}(b) = n)\ Q\times [b,B]\cap \mathcal{X}_n = \emptyset\}.\]

Let $\mathcal S = \{(Q, B)\in\mathbb P\times\mathcal H : Q\subseteq P, B\leq A\}$.  Every $\mathcal D_n$ is dense open in $\mathcal S$, so let $\mathcal D_{\infty}$ be as in Lemma \ref{dense prod fusion} and choose $(Q, B)\in\mathcal D_{\infty}$. Then,  $Q\times [0,B]\cap \bigcup_n\mathcal{X}_n =
\emptyset$: take $(x,C)\in Q\times [0,B]$ and fix arbitrary $n$.
To show that $(x,C)\not\in\mathcal{X}_n$ let $l$ be large enough
so that $\depth_B(r_l(C)) = m \geq n$. Then by construction
$Q\times [r_l(C),B]\cap \mathcal{X}_m = \emptyset$ and hence,
since $\mathcal{X}_n\subseteq \mathcal{X}_m$, we have
$(x,C)\not\in\mathcal{X}_n$. This completes the proof.
\end{proof}

Now, we borrow some terminology from \cite{Paw}: Let $\mathcal{A}$ be a family of subsets of a set $\mathcal{Z}$.
We say that $\mathcal{X}, \mathcal{Y}\subseteq\mathcal{Z}$ are
{\em compatible} (with respect to $\mathcal{A}$) if there exists
$\mathcal{W}\in \mathcal{A}$ such that
$\mathcal{W}\subseteq\mathcal{X}\cap\mathcal{Y}$. Also, we say
that $\mathcal{A}$ is $M$-{\em like} if for any
$\mathcal{B}\subseteq\mathcal{A}$  such that $|\mathcal{B}| <
|\mathcal{A}|$, every member of  $\mathcal{A}$ which is not
compatible with any member of  $\mathcal{B}$ is compatible with
$\mathcal{Z}\setminus\bigcup\mathcal{B}$.

\medskip

The families $\mathbb{P}$ of perfect subsets of
$2^{\infty}$ and $Exp(\mathcal{H})$ are $M$-like. Therefore, according to
Lemma 2.7 in \cite{Paw}, the family
$\mathbb{P}\times Exp(\mathcal{H}) = \{P\times [n,A] :
P\in\mathbb{P}\ \mbox{and}\ \   A\in\mathcal{H}\}$ is also
$M$-like. This lead us to the following:

\vspace{.25 cm}

\begin{coro}
The family of perfectly
$\mathcal H$-Ramsey subsets of $2^{\infty}\times\finkinf$ is closed under
the Souslin operation.
\end{coro}

\begin{proof}

\vspace{.25 cm}

Theorem \ref{thm1} states that the family of perfectly $\mathcal H$-Ramsey
subsets of $2^{\infty}\times\finkinf$ coincides with the family
of subsets of $2^{\infty}\times\finkinf$ which have the
$\mathbb{P}\times Exp(\mathcal{H})$-Baire property. As
pointed out in the previous paragraph, $\mathbb{P}\times
Exp(\mathcal{H})$ is $M$-like. So the proof follows from Lemma
\ref{RamseyNull} above and Lemmas 2.5 and 2.6 of
\cite{Paw} (which refer to a well-known result of Marczewski
\cite{Mar}).

\end{proof}

\section{Mathias forcing on $\finkinf$}

We recall the definition of Mathias forcing with respect to a coideal $\mathcal H$ (\cite{Ma}).
Given a coideal $\mathcal H\subseteq \ninf$,
 $$\mathbb M_{\mathcal H} = \{[a,A] : a\in [\mathbb N]^{<\infty}, A\in \mathcal H, max(a)<min(A)\},$$

with the order relation $[a,A]\leq [b,B]$ if $b$ is an initial segment of $a$, $A\subseteq B$ and $a\setminus b \subseteq B$.

If $G$ is $\mathbb M_{\mathcal H}$-generic, then  $x=\bigcup \{a: \exists A\in \mathcal H ([a,A]\in G)\}$ is said to be an $\mathbb M_{\mathcal H}$-generic real.

We define an analogous forcing notion for the space $\finkinf$.

Given $s\in\FIN_k^{<\infty}$ and $A\in\finkk$, we define the set
$$[s,A]:=\{B\in\finkk:s\sqsubseteq B\text{ and }B\leq A\}$$
where $s\sqsubseteq B$ means that $s$ is an initial segment of $B$ in its increasing order.

\begin{defn}
Let $\cH$ be a coideal in $\finkinf$. The \textbf{Mathias forcing localized at $\cH$} is the partially ordered set
$$\bM_\cH:=\left\{(s,A)\in\FIN_k^{<\infty}\times\ \cH:s\less A\right\}$$
ordered by $(s,A)\leq(t,B)$ if $[s,A]\subseteq[t,B]$, i.e., $t\sqsubseteq s$, $A\leq B$ and $s\setminus t\subseteq[B]$.
\end{defn}

 We will show some facts about this forcing notion before continuing in our study of the consistency (relative to $ZF$) of the $\cH$-Ramseyness of every subset of $\finkk$ when $\cH$ is in a suitable class of coideals. For an ultrafilter $\cU$ we use the special notation $\bM_\cU:=\bM_{\cU^\infty}$.

We say that  $\cX\in\finkinf$ is $\mathbb M_{\mathcal H}$-\textbf{generic} over a model $V$ if for every dense open subset $\mathcal D\in V$ of $\mathbb M_{\mathcal H}$, there exists a condition $(a,A)\in\mathcal D$ such that $cX\in[a,A]$.

If $G$ is a $\bM_\cH$-generic filter over $V$, then
$$\cX=\bigcup\left\{s:(\exists A\in\cH)(s,A)\in G\right\}\in(\finkk)^{V[G]}$$
is a \textbf{$\bM_\cH$-generic block sequence over $V$}.

\begin{defn}
It is said that $\mathcal H$ has the \textbf{pure decision property} (the Prikry property) if for every sentence of the forcing language $\phi$ and every condition $(a,A)\in\mathbb M_{\mathcal H}$ there exists $B\in [a,A]\cap\mathcal H$ such that $(a,B)$ decides $\phi$. 

It is said that $\mathcal H$ has the  \textbf{hereditary genericity property} (the Mathias property) if it satisfies that if $\cX$ is  $\mathbb M_{\mathcal H}$-generic over a model $V$,  then every $\cY\leq \cX$ is $\mathbb M_{\mathcal H}$-generic over $V$.
\end{defn}

The following result was proved in \cite{Ga, Ga1} for the particular case of $\mathcal H= \finkinf$.  A more general version for topological Ramsey spaces appears in \cite{DPMN}

\begin{thm}\label{prikry prop}
If $\mathcal H\subseteq\finkinf$ is a semiselective coideal then it has the pure decision property. 
\end{thm}

\begin{proof}
Suppose $\cH\subseteq\finkinf$ is a semiselective coideal, and fix a sentence $\varphi$ if the forcing language and a condition 
$(a, A)\in \mathbb M_{\mathcal H}$. For every $b\in \finkfin$ with $a\sqsubseteq b$, let 

$$\mathcal D_b = \{B\in\mathcal H\cap[\depth_A(b),A] : (b,B)\mbox{ decides } \varphi\mbox{ or}\ (\forall C\in\mathcal H\cap [b,B])\ (b,C)\mbox{ does not decide } \varphi\}.$$

\noindent and set 
$\mathcal D_b = \mathcal H\cap [\depth_A(b),A] $, for all $b\in\finkfin\upharpoonright A$ with $b\not\sqsupseteq a$. 

Each $\mathcal D_b$ is dense open in $\mathcal H\cap[\depth_A(b),A]$. Fix a diagonalization $B\in\mathcal{H}\upharpoonright   A$. 
For every $b\in\finkfin  \upharpoonright  A$. Let 

$$\mathcal F_0 = \{b\in \finkfin  \upharpoonright  B : a\sqsubseteq b\ \ \&\ (b,B) \mbox{ forces }  \varphi\},$$

$$\mathcal F_1 = \{b\in\finkfin  \upharpoonright  B : a\sqsubseteq b\ \ \&\ (b,B) \mbox{ forces }  \neg\varphi\}.$$

Let $\hat{C}\in\mathcal{H}\upharpoonright  B$ as in Lemma \ref{galvinlocal2} applied to $a$,  $B$ and $\mathcal F_0$. And let $C\in\mathcal{H}\upharpoonright  \hat{C}$ be as in Lemma \ref{galvinlocal2} applied to $a$,  $\hat{C}$ and $\mathcal F_1$. Let us prove that $(a,C)$ decides $\varphi$. So let $(b_0,C_0)$ and $(b_1,C_1)$ be two different arbitrary extensions of $(a,C)$. Suppose that $(b_0,C_0)$ forces $\varphi$ and $(b_1,C_1)$ forces $\neg\varphi$. Then $b_0\in\mathcal F_0$ and $b_1\in\mathcal F_1$. But $b_0,b_1\in\finkfin\upharpoonright C$, so by the choice of $C$ this means that every element of $\mathcal H\cap[a,C]$ has an initial segment  in $\mathcal F_0$ and an initial segment in $\mathcal F_1$. So there exist two compatible extensions of $(a,C)$ such that one forces $\varphi$ and the other forces $\neg\varphi$. A contradiction. So either both $(b_0,C_0)$ and $(b_1,C_1)$ force $\varphi$ or both $(b_0,C_0)$ and $(b_1,C_1)$ force $\neg\varphi$. Therefore $(a,C)$ decides $\varphi$.

\end{proof}

Now we will prove  that if $\mathcal H\subseteq\finkinf$ is semiselective then it has the hereditary genericity property (see Theorem \ref{mathias prop} below). 

Given a selective ultrafilter $\mathcal U\subset\finkinf$, let $\mathbb M_{\mathcal U}$ be set of all pairs $(a,A)$ such that $A\in\mathcal U$ and $[a,A]\neq\emptyset$. Order $\mathbb M_{\mathcal U}$ with the same ordering used before.

\bigskip
Extending to the context of $\FIN_k$ the notion of capturing devised by Mathias in \cite{Ma} is essential in what follows.

\begin{defn}\label{captures}
 Let $\mathcal U\subseteq\finkinf$ be a selective ultrafilter, $\mathcal D$ a dense open subset of $\mathbb M_{\mathcal U}$, and $a\in\finkfin$. We say that $A$ {\bf captures} $(a,\mathcal D)$ if $A\in\mathcal U$, $[a,A]\neq\emptyset$, and for all $B\in [a,A]$ there exists $m>|a|$ such that $(r_m(B),A)\in\mathcal D$.
\end{defn}
 
\begin{lem}\label{lemma captures}
Let $\mathcal U\subseteq\finkinf$ be a selective ultrafilter and $\mathcal D$ a dense open subset of $\mathbb M_{\mathcal U}$. Then, for every $a\in\finkfin$ there exists $A\in \mathcal U$ which captures $(a,\mathcal D)$.
\end{lem}
\begin{proof}

Given $a\in\finkfin$, we can choose $B\in\mathcal U$ such that $[a,B]\neq\emptyset$, for example $\one$. We  define a collection $(C_b)_{b\in\finkfin\upharpoonright B}$ with $[b,C_b]\neq\emptyset$, such that:

\begin{enumerate}
\item For all $b_1,b_2\in\finkfin\upharpoonright  B$, if $\depth_B(b_1)=\depth_B(b_2)$ then  $C_{b_1} = C_{b_2}$.
\item For all $b_1,b_2\in\finkfin\upharpoonright  B$, if $b_1\sqsubseteq b_2$ then  $C_{b_1} \geq C_{b_2}$.
\item For all $b\in\finkfin\upharpoonright  B$ with $a\sqsubseteq b$ either $(b,C_b)\in\mathcal D$ or if such a $C_b\in \mathcal D$ does not exist then $C_b=B$. 
\end{enumerate}

For every $b\in\finkfin\upharpoonright  B$, let $C_n=C_b$ if $\depth_B(b)=n$. Notice that $C_n\geq C_{n+1}$, for all every $n\in\mathbb N$. By selectivity, let $C\in\mathcal U\cap[a,B]$ be a diagonalization of $(C_n)_{n\in\mathbb N}$. Then, for all $b\in\finkfin\upharpoonright  C$ with $a\sqsubseteq b$, if there exists a $\hat{C}\in\mathcal U$ such that $(b,\hat{C})\in\mathcal D$, we must 
have $(b,C)\in\mathcal D$. 

Let $\mathcal X=\{D\in\finkinf : D\leq C\rightarrow (\exists b\in\finkfin\upharpoonright  D)\ a\sqsubset b\ \&\ (b,C)\in\mathcal D\}$. $\mathcal X$ is a metric open subset of $\finkinf$ and therefore, by Lemma \ref{openram}, it is $\mathcal U$-Ramsey. Take $\hat{C}\in\mathcal U\cap[\depth_C(a),C]$ such that $[a,\hat{C}]\subseteq\mathcal X$ or $[a,\hat{C}]\cap\mathcal X=\emptyset$. We will show that the first alternative holds: Pick $A\in\mathcal U\cap[a,\hat{C}]$ and $(a',A')\in\mathcal D$ such that $(a',A')\leq (a,A)$. Notice that $a\sqsubseteq a'$ and therefore, by (3), we have $(a',C)\in\mathcal D$. By the definition of $\mathcal X$, we also have $A'\in\mathcal X$. Now choose $A''\in\mathcal U\cap[a',A']$. Then $(a', A'')$ is also in $\mathcal D$ and therefore $A''\in\mathcal X$. But $A''\in[a', A']\subseteq[a, A]\subseteq [a,\hat{C}]$. This implies $[a,\hat{C}]\subseteq\mathcal X$. Finally, that $A$ captures $(a,\mathcal D)$ follows from the definition of $\mathcal X$ and the fact that $[a,A]\subseteq [a,\hat{C}]\subseteq [a,C]$. This completes the proof.
\end{proof}

\begin{thm}\label{forcing M_U}
Let $\mathcal U\subseteq \finkinf$ be a selective ultrafilter in a given transitive model $V$ of $ZF+DC\mathbb R$. Forcing over $V$ with $\mathbb M_{\mathcal U}$ adds a generic $g\in\finkinf$ with the property that $g\leq^*A$ for all $A\in\mathcal U$. In fact, $B\in\finkinf$ is $\mathbb M_{\mathcal U}$-generic over $V$ if and only if $B\leq^* A$ for all $A\in\mathcal U$. %Also, $V[\mathcal U][g]=V[g]$.
\end{thm}
\begin{proof}

Suppose that $B\in \finkinf$ is $\mathbb M_{\mathcal U}$-generic over $V$. Fix an arbitrary $A\in\mathcal U$. Let us show that the set 
$\{(c,C)\in\mathbb M_{\mathcal U} : C\leq^*A\}$ 
is dense open. Note that this set is in $V$. Fix $(a,A')\in\mathbb M_{\mathcal U}$. 
Since $\mathcal U$ admits finite changes, choose $A''\leq^* A$ in $\mathcal U$ such that $[a, A'']\neq \emptyset$. Clearly, $[a,A']\cap [a,A'']\neq \emptyset$, and then there is 
$n\in \mathbb N$ and $C_1\in \mathcal U$ such that $[n,C_1]\subseteq [a, A']\cap [a,A'']$. Let $c=r_n(C_1)$.
By ${\bf A3}$ mod $\mathcal U$, there exists $C_2\in\mathcal U\cap[\depth_{A'}(c),A']$ such that $\emptyset\neq[c,C_2]\subseteq[c,C_1]$. It is clear that $[c,C_2]\subseteq[c,A'']$ and therefore $C_2\leq^*A''\leq^* A$. Also, since $\depth_{A'}(c)\geq\depth_{A'}(a)$, we have $[a,C_2]\neq\emptyset$. Thus,  $(a,C_2)\leq (a,A')$. That is, $\mathcal D$ is dense. It is obviously open. So, by genericity, there exists  $(c,C)\in\mathcal D$ such that $B\in[c,C]$. Hence $B\leq^*A$.

Now, suppose that $B\in\finkinf$ is such that $B\leq^*A$ for all $A\in\mathcal U$, and let $\mathcal D$ be a dense open subset of $\mathbb M_{\mathcal U}$. We need to find $(a,A)\in\mathcal D$ such that $B\in[a,A]$. In $V$, by using Lemma \ref{lemma captures} iteratively, we can define a sequence $(A_n)_n$ such that $A_n\in\mathcal U$, $A_{n+1}\leq A_n$, and $A_n$ captures $(r_n(B),\mathcal D)$. Since $\mathcal U$ is in $V$ and selective, we can choose $A\in\mathcal U$, in $V$, such that $A\leq^*A_n$ for all $n$. By our assumption on $B$, we have $B\leq^*A$. So there exists an $a\in \finkfin $ such that $[a,B]\subseteq[a,A]$. Let $m=\depth_B(a)$. By ${\bf A3}$ mod $\mathcal U$, we can assume that $a=r_m(B)=r_m(A)$, and also that $A\in[r_m(B),A_m]$. Therefore, $B\in[m,A]$ and $A$ captures $(r_m(B),\mathcal D)$. Hence, the following is true in $V$: 

\begin{equation}\label{eq captures}
(\forall C\in[m,A]) (\exists n>m) ((r_n(C),A)\in\mathcal D).
\end{equation}

Let $\mathcal F=\{b : (\exists n>m)(b\in r_n[m,A]\ \&\ (b,A)\notin\mathcal D)\}$ and give $\mathcal F$ the strict end-extension ordering $\sqsubset$. Then the relation $(\mathcal F,\sqsubset)$ is in $V$, and by equation \ref{eq captures} $(\mathcal F,\sqsubset)$ is well-founded. Therefore, by a well-known argument due to Mostowski, equation \ref{eq captures} holds in the universe. Hence, since $B\in[m,A]$, there exists $n>m$ such that $(r_n(B),A)\in\mathcal D$. But $B\in[r_n(B),A]$, so $B$ is $\mathbb M_{\mathcal U}$--generic over $V$.
\end{proof}

\begin{coro} \label{Mu generic}
If $B$ is $\mathbb M_{\mathcal U}$--generic over some model $V$ and $A\leq B$ then $A$ is also $\mathbb M_{\mathcal U}$--generic over $V$. In other words, $\cU$ has the hereditary genericity property.
\end{coro}

\begin{lem}\label{Mu iteration}
Let $\mathcal H\subseteq\finkinf$ be a semiselective coideal. Consider the forcing notion $\mathbb P = (\mathcal H,\leq^*)$ and let $\hat{\mathcal U}$ be a $\mathbb P$--name for a $\mathbb P$--generic ultrafilter. Then the iteration $\mathbb P *\mathbb M_{\hat{\mathcal U}}$ is equivalent to the forcing $\mathbb M_{\mathcal H}$.
\end{lem}
\begin{proof}

Recall that $\mathbb P *\mathbb M_{\hat{\mathcal U}}=\{(B,(\dot{a},\dot{A})) : B\in\mathcal H\ \&\ B\vdash (\dot{a},\dot{A})\in \mathbb M_{\hat{\mathcal U}}\}$, with the ordering  $(B,(\dot{a},\dot{A}))\leq (B_0,(\dot{a}_0,\dot{A}_0)) \Leftrightarrow B\leq^* B_0 \ \&\ (\dot{a},\dot{A})\leq (B_0,(\dot{a}_0,\dot{A}_0)$. The mapping $(a,A) \rightarrow (A, (\hat{a}, \hat{A}))$ is a dense embedding  from $\mathbb M_{\mathcal H}$ to $\mathbb P *\mathbb M_{\hat{\mathcal U}}$ (here $\hat{a}$ and $\hat{A}$ are the canonical $\mathbb P$-names for $a$ and $A$, respectively): It is easy to show that this mapping preserves the order. So, given $(B,(\dot{a},\dot{A}))\in\mathbb P *\mathbb M_{\hat{\mathcal U}}$, we need to find $(d,D)\in\mathbb M_{\mathcal H}$ such that $(D, (\hat{d}, \hat{D}))\leq (B,(\dot{a},\dot{A}))$. Since $\mathbb P$ is $\sigma$-distributive, there exists $a\in \finkfin $, $A\in\mathcal H$ and $C\leq* B$ in $\mathcal H$ such that $C\vdash_{\mathbb P} (\hat{a}=\dot{a}\ \&\ \hat{A}=\dot{A})$ (so we can assume $a\in \finkfin\upharpoonright  C$). Notice that $(C, (\hat{a}, \hat{A}))\in\mathbb P *\mathbb M_{\hat{\mathcal U}}$ and $(C, (\hat{a}, \hat{A})), (C, (\hat{a}, \hat{A}))\leq (B,(\dot{a},\dot{A}))$. So, $C\vdash_{\mathbb P} \hat{C}\in\hat{\mathcal U}$ and $C\vdash_{\mathbb P} \hat{A}\in\hat{\mathcal U}$. Then, $C\vdash_{\mathbb P} (\exists x\in\hat{\mathcal U})(x\in[\hat{a},\hat{A}]\ \&\ x\in[\hat{a},\hat{C}]$. So there exists $D\in\mathcal H$ such that $D\in[a,A]\cap[a,C]$. Hence, $(D, (\hat{a}, \hat{D}))\leq (B,(\dot{a},\dot{A}))$. This completes the proof.
\end{proof}

\bigskip

The next theorem follows inmediately from  Corollary \ref{Mu generic}  and Lemma \ref{Mu iteration}.

\begin{thm}\label{mathias prop}
If $\mathcal H\subseteq\finkinf$ is a semiselective coideal then it has the hereditary genericity  property.
\end{thm}

For the next lemma it will be useful to have the following notion. Given $P\in \finkinf$, $P=\langle p_0, p_1, \dots \rangle$, every element of $\finkfin\upharpoonright X$ is obtained from a finite subsequence of $X$ by
$$T^{(i_0)}(p_{n_0})+ \dots +T^{(i_l)}(p_{n_l})$$
for some increasing  sequence $n_0<\dots <n_l$  and some choice $i_0, \dots , i_l\in \{0, 1, \dots , k\}$ with at least one of the numbers 
$i_0, \dots , i_l$ equal to $0$.
We can thus define a well ordering of $\finkfin\upharpoonright X$ in the followng way.

Let $s=T^{(i_0)}(p_{n_0})+ \dots +T^{(i_l)}(p_{n_l})$ and $t=T^{(j_0)}(p_{m_0})+ \dots +T^{(j_h)}(p_{m_h})$, 
then 

\[
s<_{lex}t \text{ if } 
\begin{cases}
\langle n_0, \dots, n_l\rangle <_{lex} \langle m_0, \dots , m_h\rangle,

 \text{or} \\
\langle n_0, \dots, n_l\rangle = \langle m_0, \dots , m_h\rangle \text{ and }
\langle i_0, \dots , i_l\rangle <_{lex} \langle j_0, \dots , j_h\rangle.
\end{cases}
\]

\begin{lem}
(\cite{Fa} for coideals in the space $\ninf$)
For a coideal $\mathcal H$ on $\finkinf$, the following are equivalent:
\begin{enumerate}
\item $\mathcal H$ is semiselective,
\item $\mathbb M_{\mathcal H}$ has the pure decision property,
\item $\mathbb M_{\mathcal H}$ has the hereditary genericity property.
\end{enumerate}
\end{lem}

\begin{proof}
We prove first $(2)$ implies $(1)$.
Suppose $\mathcal H$ is not semiselective. Then
there is $A\in \mathcal H$ and  a sequence $(\mathcal D_a: a\in \finkfin\upharpoonright A)$ of dense open subsets of $(\mathcal H, \leq^*)$ such that no element of $\mathcal{H}\upharpoonright A$ is a diagonalization of the sequence.

We now work with $\{B: B\leq A\}$, and consider the forcing notion $\mathbb M_{\mathcal{H}\upharpoonright A}$.

Pick a sequence $(\mathcal  A_a : a\in \finkfin\upharpoonright A)$ of maximal antichains  of
$(\mathcal{H}\upharpoonright A, \leq^*)$ with $\mathcal A_a \subseteq \mathcal D_a$,  such that no element of $\mathcal{H}\upharpoonright A$
is a diagonalization of the sequence of dense open sets determined by the antichains. 

%We can assume without loss of generality
%that for every $n$,  $\mathcal A_{n+1}$ refines $\mathcal A_n$; and also that $A=\mathbb N$.

Each antichain $\mathcal A_a$ determines a maximal antichain in $\mathbb M_{\mathcal{H}\upharpoonright A}$, namely,
$\{[\emptyset, A]: A\in \mathcal A_a\}$.
Let $\dot x$ be a canonical $\mathbb M_{\mathcal
H\upharpoonright A}$-name for a  generic sequence, and let $\tau_a$ be a name of the
unique element of the antichain $\mathcal A_a$ such that $\dot
x\leq^* \tau_a$ is forced. 

Let us show that $\dot x$ is forced not to be a diagonalization of $(\mathcal A_a)$. Suppose the contrary, and let $[s,B]$ be a condition that forces that for every 
$a\in \finkfin\upharpoonright \dot x$, $[a,\dot x]\subseteq [a, \tau_a]$.
Now, $B/s$ is not a diagonalization of the sequence $(\mathcal A_a)$ and so there is $t\in B/s$ be such that $[t, B/t]$ is not contained in $[t,A]$ for any element $A$ 
of $\mathcal A_t$. There is an element $A_t$ of $\mathcal A_t$ such that there is $C\in \mathcal H$ with $C\leq B$ and $C\leq A_t$.
Since $[t, B/t]$ is not contained in $[t,A_t]$, there is some $r\in \finkfin\upharpoonright B/t$ not in $\finkfin\upharpoonright A_t$.
The condition $[s\frown t \frown r, B]$ forces that $\tau_t=A_t$, and forces that $[t, \dot x/t ]$ s not contained in $[t, \tau_t]$. But this contradicts that being an extension of $[s, B]$, 
it must force $[t, \dot x/t ]\subseteq [t, \tau_t]$.

Consider the formula

\medskip
$\phi$ : 
$\dot s$ is the $<_{lex}$-first element of $\finkfin\upharpoonright \dot x$ such that $\dot x/s\not\leq \tau_s$, and the number of blocks  of $\dot x$ below $\dot s$ is even.

\medskip
Since $\dot x$ is forced not diagonalize the sequence $(\mathcal
A_a)_a$, $\dot s$ is well defined.

If $\mathbb M_{\mathcal H}$ has the pure decision property, there is
$\bar A\in \mathcal H$ such that $[\emptyset, \bar A]$ decides $\phi$.

Now, we find    $s_1<t_1<s_2<t_2$ in $\finkfin$ as follows:

$s_1$ is the least $s\in \finkfin\upharpoonright \bar A$ such that $[s, \bar A]$ is not contained in $[s, X]$ for any
member $X$ of $\mathcal A_{s_1}$. Such $s_1$ exists because $\bar A$ is not a diagonalization of $\mathcal A_{s_1}$.

Let $A_1\in \mathcal A_{s_1}$  be such that  there is $C\in \mathcal H$ with $C\leq \bar A$ and $C\leq A_1$, and pick  $t_1>s_1$ such that $t_1\in (\finkfin\upharpoonright A) \setminus (\finkfin\upharpoonright A_1)$.

Set now $A'\in \mathcal H$ such that $A'\leq \bar A$ and $A'\leq A_1/ t_1$. Since $A'$
does not diagonalize the sequence $(\mathcal A_a)_a$, let $s_2$ be
the least $s\in \finkfin \upharpoonright A'$ such that $[s_2, A']$ is not contained in $[s_2, X]$ for any
element $X$ of $\mathcal A_{s_2}$. 
Take $A_2\in \mathcal A_{s_2}$ such that  there is $C\in \mathcal H$ with $C\leq A'$ and $C\leq A_2$, and pick  $t_2>s_2$ such that $t_2\in (\finfin\upharpoonright A' ) \setminus (\finfin\upharpoonright A_2)$.

 Take  $A''\in \mathcal H$ below $A'$ and below $ A_2/t_2$.

Then, $$[\{s_1,  t_1\}, A'] \Vdash \dot s= s_1 \mbox{ and } \dot x \mbox{ does not have elements below } s.$$

So, this contition forces  $\phi$, while
$$[\{s_1, s_2, t_2\}, A''] \Vdash \dot s = s_2 \mbox{ and } \dot x \mbox{ has one element below } s. $$
thus, this second condition forces $\neg \phi$, a contradiction since both conditions extend $[\emptyset , A]$.

\medskip
Now, we prove $(3)$ implies$ (1)$.

As in the previous case, suppose $\mathcal H$ is not semiselective and thus
there is $A\in \mathcal H$ and  a sequence $(\mathcal D_a: a\in \FIN_k\upharpoonright A)$ of dense open subsets of $(\mathcal H, \leq^*)$ such that 
no element of $\mathcal{H}\upharpoonright A$ is a diagonalization of the sequence.
Pick a sequence $(\mathcal  A_a : a\in \FIN_k\upharpoonright A)$ of maximal antichains  of
$(\mathcal{H}\upharpoonright A, \subseteq^*)$ with $\mathcal A_a \subseteq \mathcal D_a$,  such that no element of $\mathcal{H}\upharpoonright A$
is a diagonalization of the sequence. 

We know that  $\dot x$ is forced not to be a diagonalization of the sequence $(\mathcal A_a)$. But since for every $a$ $\dot x\subseteq^* \tau_a$ is forced. 
 there is $y\leq \dot x$ which diagonalizes $(A_a)$ and thus this $y$ cannot be in $\mathcal H$.
\end{proof}

\begin{thm}
Suppose $\lambda$ is a Mahlo cardinal and let $V[G]$ be a generic extension by $Col(\omega, \lambda)$.
If $\mathcal H$ is a semiselctive coideal on $\finkinf$ in $V[G]$, then every subset of $\finkinf$ in $L(\mathbb R)$ of $V[G]$ is $\mathcal H$-Ramsey.
\end{thm}

\begin{proof}
Let $\mathcal H\subseteq \finkinf$ be a semiselective coideal in $V[G]$; and let $\Mh$ be the coresponding forcing with respect to
$\mathcal H$. Let $\mathcal A \subseteq \finkinf$  in $L(\mathbb R)^{V[G]}$; in particular,
$\mathcal A$ is defined in $V[G]$ by a formula $\varphi$ from a
sequence of ordinals.

Let  $\dot{\mathcal H}$ be a $\col$-name for
$\mathcal H$. Notice that $\dot{\mathcal H}\subseteq V_\lambda$; also,  elements of $\finkinf$ in $V[G]$
 have  $\col$-names in $V_\lambda$.
Therefore $\dot{\mathcal H}$ is interpreted in $V_\lambda[G]$ as $\mathcal H$.

Since $\lambda$ is a Mahlo  cardinal, the set of inaccessible
cardinals below $\lambda$ is stationary, and we can find an
inaccessible $\kappa<\lambda$ such that

$$\langle V_\kappa, \in , \dot{\mathcal H}\cap V_\kappa, Col(\omega, <\lambda)\cap V_\kappa\rangle \prec \langle V_\lambda, \in ,\dot{\mathcal H}, Col(\omega, < \lambda)\rangle.$$

It follows that
$$\langle V_\kappa[G_\kappa],  \in ,\mathcal H\cap V_\kappa[G_\kappa], Col(\omega, <\lambda)\cap V_\kappa[G_\kappa]\rangle \prec \langle V_\lambda[G], \in ,\mathcal H, Col(\omega, < \lambda)\rangle,$$

where $G_\kappa = G\cap Col(\omega, <\kappa)$.
This can be verified as follows. Let $\phi(\tau_1,\dots, \tau_n)$
be a formula of the forcing language of $Col(\omega, <\kappa)$. If
this formula is valid in $ V_\kappa[G_\kappa]$, interpreting
$\tau_1, \dots, \tau_n$  by $G_\kappa$, then there is a condition
$p\in G_\kappa$ that forces $\phi$. Since $p$ is also in
$G$ and $\tau_1, \dots, \tau_n$  are also
$Col(\omega,<\lambda)$-names,  $\phi$ is also valid in
$V_\lambda[G]$. Conversely, if $\phi(\tau_1, \dots,
\tau_n)$ is valid in $V_\lambda[G]$ there is $q\in
G$ that forces it (in $V_\lambda$). Since the names
$\tau_1, \dots, \tau_n$ are $Col(\omega,<\kappa)$-names,
$p=q\restriction (\omega\times\kappa)$ forces in $V_\lambda$ the
formula $\phi$. By elementarity this also holds in $V_\kappa$, and
therefore $\phi$ interpreted by $G_\kappa$ is valid in
$V_\kappa[G_\kappa]$.

We want to show that the coideal $\mathcal H\cap V_\kappa[G_\kappa]$ is semiselective in $V_\kappa[G_\kappa]$.
We will show that in the model $V_\kappa[G_\kappa]$ the partial order $\mathbb M_{\mathcal H\cap V_\kappa[G_\kappa]}$ has the pure decision property.

We now consider the forcing notion $\mathbb M_{\mathcal H\cap V_{\kappa [G_\kappa]}}$ in $V_\kappa[G_\kappa]$.
Let $\psi$ be a formula in the forcing language of  $\mathbb M_{\mathcal H\cap V_{\kappa[G_\kappa]}}$, and
let $[a,A]$  be a condition in $\mathbb M_{\mathcal H\cap V_{\kappa [G_\kappa]}}$. Since this is also a condition in
$\mathbb M_{\mathcal H}$, in $V_\lambda[G_\lambda]$ there is $B\in \mathcal H$, $B\subseteq A$, such that
$[a,B]$ decides $\psi$.
But by elementarily this statement also holds in $V_\kappa[G_\kappa]$ and so in $V_\kappa[G_\kappa]$ there is
$B\in \mathcal H\cap V_\kappa[G_\kappa]$, $B\subseteq A$ such that $[a,B]$ decides $\psi$.

%This shows that (in $V_\kappa[G_\kappa]$) the forcing $\mathbb M_{\mathcal H\cap V_\kappa[G_\kappa]}$ has the Prikry property, and by Lemma \ref{equi} it also has the Mathias property and it is semiselective.

Since this happens for every formula of the forcing language of $\Mh\cap V_\kappa$, then the coideal $\mathcal H\cap V_\kappa[G_\kappa]$ has the pure decision property and thus it is semiselective and it also has the hereditary genericity property.

Now the proof can be finished as in \cite{Ma}. Let $\dot x$ be the
canonical name of a $\mathbb M_{\mathcal H \cap V[G_\kappa]}$
generic sequence, and consider the formula $ \varphi( \dot x)$ in the
forcing language of $V[G_\kappa]$. By the pure decision property of
$\mathcal H \cap V[G_\kappa]$, there is $A'\leq A$, $A'\in
\mathcal H \cap V[G_\kappa]$, such that $[a,A']$ decides $\varphi(
\dot x)$. Since $2^{2^{FIN}}$ computed in $V[G_\kappa] $ is
countable in $V[G]$, there is (in $V[G]$) a  $\mathbb M_{\mathcal H
\cap V[G_\kappa]}$-generic  sequence $x$ over $V[G_\kappa]$ such that
$x\in [a,A']$. To see that there is such a generic sequence in $\mathcal
H$ we argue as in 5.5 of \cite{Ma}  using the semiselectivity of
$\mathcal H$ and the fact that $\mathcal H\cap V[G_\kappa]$ is
countable in $V[G]$ to obtain an element of $\mathcal H$ which is
generic.
 By the hereditary genericity property
of $\mathcal H \cap V[G_\kappa]$, every  $y\in [a,x\setminus a]$ is
also $\mathbb M_{\mathcal H \cap V[G_\kappa]}$-generic over
$V[G_\kappa]$, and also $y\in [a,A']$. Thus $\varphi(x)$ if and only
if $[a,A']\Vdash\varphi(\dot x)$, if and only if
 $\varphi (y)$.
Therefore, $[a,x\setminus a]$ is contained in $\mathcal A$ or is
disjoint from $\mathcal A$.
\end{proof}

As in \cite{Ma}, we obtain the following.

\begin{coro}
If $ZFC$ is consistent with the existence of a Mahlo
cardinal, then so is the statement that for every semiselective
co-ideal $\mathcal H$  on $\finkinf$, every subset of $\finkinf$ in $L(\mathbb R)$
is $\mathcal H$-Ramsey.
\end{coro}

\end{document}